\pdfoutput=1
\RequirePackage{ifpdf}
\ifpdf 
\documentclass[pdftex]{sigma}
\else
\documentclass{sigma}
\fi

\usepackage{multirow,longtable}
\usepackage{tikz-cd}

\newcommand\car[1] {\langle#1\rangle}

\newcommand{\g}{\mathfrak{g}} 	
\newcommand{\Z}{\mathbb{Z}} 	
\newcommand{\N}{\mathbb{N}} 	
\newcommand{\C}{\mathbb{C}} 	
\newcommand{\F}{\mathbb{F}} 	

\numberwithin{equation}{section}

\newtheorem{Theorem}{Theorem}[section]
\newtheorem*{theoremX}{Theorem~\ref{thm:solutionsgrpeq}}
\newtheorem*{lemmaX}{Lemma~\ref{lm_transparent}}
\newtheorem{Corollary}[Theorem]{Corollary}
\newtheorem{Lemma}[Theorem]{Lemma}
\newtheorem{Question}[Theorem]{Question}
 { \theoremstyle{definition}
\newtheorem{Definition}[Theorem]{Definition}
\newtheorem{Example}[Theorem]{Example}}

\begin{document}

\allowdisplaybreaks

\newcommand{\arXivNumber}{1612.07960}

\renewcommand{\PaperNumber}{076}

\FirstPageHeading

\ShortArticleName{Factorizable $R$-Matrices for Small Quantum Groups}

\ArticleName{Factorizable $\boldsymbol{R}$-Matrices for Small Quantum Groups}

\Author{Simon LENTNER and Tobias OHRMANN}
\AuthorNameForHeading{S.~Lentner and T.~Ohrmann}
\Address{Fachbereich Mathematik, University of Hamburg,\\ Bundesstra{\ss}e 55, 20146 Hamburg, Germany}
\Email{\href{mailto:simon.lentner@uni-hamburg.de}{simon.lentner@uni-hamburg.de}, \href{mailto:tobias.ohrmann@uni-hamburg.de}{tobias.ohrmann@uni-hamburg.de}}

\ArticleDates{Received January 16, 2017, in f\/inal form September 15, 2017; Published online September 25, 2017}

\vspace{-1mm}

\Abstract{Representations of small quantum groups $u_q({\mathfrak{g}})$ at a root of unity and their extensions provide interesting tensor categories, that appear in dif\/ferent areas of algebra and mathematical physics. There is an ansatz by Lusztig to endow these categories with the structure of a braided tensor category. In this article we determine all solutions to this ansatz that lead to a non-degenerate braiding. Particularly interesting are cases where the order of $q$ has common divisors with root lengths. In this way we produce familiar and unfamiliar series of (non-semisimple) modular tensor categories. In the degenerate cases we determine the group of so-called transparent objects for further use.}


\Classification{17B37; 20G42; 81R50; 18D10}

\vspace{-2.5mm}

\section{Introduction}

Hopf algebras with $R$-matrices, so called \emph{quasitriangular Hopf algebras}, give rise to tensor categories with a braiding $c\colon V\otimes W\stackrel{\sim}{\longrightarrow} W\otimes V$. Of particular interest are braided tensor categories where the braiding fulf\/ills a certain
non-degeneracy condition, see Def\/inition~\ref{def:FactorizabilityOfCats}, which is equivalent to the fact that there are no
\emph{transparent objects}~$V$, i.e., no objects where the double-braiding $c^2\colon V\otimes W\stackrel{\sim}{\longrightarrow} V\otimes W$ is the identity for all~$W$. A $\C$-linear tensor category with a nondegenerate braiding, as well as f\/initeness conditions and another natural transformation $\theta\colon V\stackrel{\sim}{\longrightarrow} V$ (twist), is called a~\emph{modular tensor category}. Note that we do not require the category to be semisimple.

Modular tensor categories have many interesting applications: They give rise to topological invariants and mapping class group actions \cite{KL01, Tur94}. For example, the standard genera\-tors~$T$,~$S$ of the mapping class group of the torus~$\mathrm{SL}_2(\Z)$ are constructed from~$\theta$ and~$c^2$, respectively. A~dif\/ferent source for modular tensor categories in mathematical physics are vertex algebras. There are only few example classes of modular tensor categories, in particular non-semisimple ones.

The aim of the present article is to provide modular tensor categories from \emph{small quantum groups} $u_q(\g)$ at a primitive $\ell$-th root of unity $q$ for a f\/inite-dimensional simple complex Lie algebra~$\g$. Lusztig \cite{Lus90} has constructed these f\/inite-dimensional Hopf algebras and provided an ansatz for an $R$-matrix $R_0\bar{\Theta}$, where the f\/ixed element $\bar{\Theta} \in u_q(\g)^{-} \otimes u_q(\g)^{+}$ is constructed from a dual basis of PBW generators, while $R_0 \in u_q(\g)^0 \otimes u_q(\g)^0$ is a free parameter subject to some constraints. He gives one canonical solution for~$R_0$ whenever~$\ell$ has no common divisors with root lengths, otherwise there are cases where no $R$-matrix exists~\cite{KS11} and the quantum group becomes more interesting \cite{Len14}, involving, e.g., the dual root system. Of particular interest in conformal f\/ield theory \cite{FGST06, FT10, GR15} is the most extreme case where all root lengths $(\alpha,\alpha)$ divide~$\ell$. In particular our article addresses the question, which modular tensor category appear in these cases. We f\/ind indeed, e.g., in Lemma~\ref{lm:Lusztigkernel} that these extremal cases give especially nice $R$-matrices; although in general they are not factorizable and will require modularization (see for example~\cite{Bru00}) to match the CFT side.

But even if there are no common divisors with the root length, the resulting braided tensor categories may not fulf\/ill the non-degeneracy condition and hence provides no modular tensor category.

\looseness=-1 Both obstacles (existence and non-degeneracy) can be be resolved by extending the Cartan part of the quantum group by a choice of a lattice $\Lambda_R\subseteq \Lambda \subseteq \Lambda_W$ between root- and weight-lattice, respectively a choice of a subgroup of the fundamental group $\pi_1:=\Lambda_W/\Lambda_R$, corresponding to a choice of a Lie group between adjoint and simply-connected form. These extensions are already present in~\cite{Lus90} as the choice of two lattices~$X$,~$Y$ with pairing $X\times Y\to \C^\times$ (root datum). In this way the number of possible $R$ matrices increases and the purpose of our paper is to study them all.

In a previous article \cite{LN14b} we have already constructed some solutions $R_0$ in this spirit (under some additional assumptions). As it turns out, the solutions can be parametrized by subgroups $H_1,H_2\subset \pi_1$ and group pairings between $H_1$, $H_2$, and the set of solutions depends on the common divisors of $\ell$ not only with root lengths, but also divisors of the Cartan matrix. Some cases admit no braided structure, while others have multiple in-equivalent solutions. An interesting occurrence was for example that $B_n$ behaves dif\/ferently for~$n$ odd or even, and that $D_{2n}$ with non-cyclic fundamental group allows several more solutions with non-symmetric~$R_0$.

In the present article we conclude this ef\/fort: First we introduce more systematical techniques that allow us to compute a list of all quasitriangular structures (without additional assumptions, so we f\/ind more solutions). Then our new techniques allow us to determine, which of these choices fulf\/ill the non-degeneracy condition. We also determine which cases have a ribbon structure. A~main role in the f\/irst part is played by a natural pairing~$a_\ell$ on the fundamental group~$\pi_1$ which depends only on the common divisors of~$\ell$ with the fundamental group and encapsulates the essential $\ell$-dependence. Then the non-degeneracy of the braiding turns out to depend only on the $2$-torsion of the abelian group in question.

Our result produces a list of modular tensor categories for representations of quantum groups. Moreover we use our methods to explicitly describe the group of transparent objects if the category is not modular, which is for example a prerequisite for modularization.

We now discuss our methods and results in more detail:

In Section~\ref{section2} we brief\/ly recall the Lie theory and Hopf algebra preliminaries: For every f\/inite-dimensional (semi-)simple complex Lie algebra $\g$ and a~primitive $\ell$-th root of unity $q$ Lusztig has introduced in~\cite{Lus90} the \emph{small quantum group} $u_q(\g)$ which has a~triangular decomposition $u_q^+u_q^0u_q^-$ where the (exponentiated) \emph{Cartan algebra} $u_q^0$ is the groupring of the root lattice $\Lambda_R$ modulo some suitable sublattice and $u_q^\pm$ are generated by \emph{simple root vectors} $E_{\alpha_i}$, $F_{\alpha_i}$ fulf\/illing $q$-deformed Serre relations. In \cite[Section~32]{Lus93} he gives an ansatz for an $R$-matrix in the form~$R_0\bar{\Theta}$, where~$\bar{\Theta}$ consists of dual PBW basis' and $R_0\in u_q^0\otimes u_q^0$ is an arbitrary element in the Cartan part that has to fulf\/ill certain relations.

Our goal is to study the existence and non-degeneracy of $R$-matrices of this form for the quan\-tum group $u_q(\g,\Lambda,\Lambda')$ with any choice of lattice between root- and weight-lattice \smash{$\Lambda_R\subseteq \Lambda \subseteq \Lambda_W$} and any possible choice of quotient by a subgroup $\Lambda'\subseteq \Lambda$ in the Cartan part $u^0=\C[\Lambda/\Lambda']$. Later, we prove that $\Lambda'$ is in fact unique if we want a quasitriagular structure (Corollary~\ref{cor:LambdaPrime}).

The $R_0$-matrix has the following interpretation: It is an $R$-matrix for the groupring $\C[\Lambda/\Lambda']$ and it appears as the braiding between highest-weight vectors in our $u_q(\g)$-modules. Thus the previous theorem clarif\/ies which choices for an $R$-matrix for the group ring lift to the quantum group.

In Section~\ref{section3} we address the question of constructing quasi-triangular $R$-matrices. First we brief\/ly recall the following general combinatorial result in~\cite{LN14b}:
\begin{theoremX}
The $R_0$-matrix is necessarily of the form
\begin{gather*}
f(\mu,\nu)=\frac 1d q^{-(\mu, \nu)}g(\bar{\mu},\bar{\nu})\delta_{\bar{\mu}\in H_1}\delta_{\bar{\nu}\in H_2},
\end{gather*}
where $H_1$, $H_2$ are subgroups of $\Lambda/\Lambda_R \subseteq \pi_1$ with $|H_1|=|H_2|=:d$ $($not necessarily isomorphic!$)$ and $g\colon H_1\times H_2\to\C^{\times}$ is a pairing of groups.
\end{theoremX}

Then we proceed differently than in the previous article: Using the previous result, we prove in Lemma~\ref{lm:NondegGroupPairing} that the quasitriangularity of~$R$ is equivalent to the assertion that the group pairing $\hat{f}:=|\Lambda/\Lambda'|\cdot f$ between the preimages $G_i:=\Lambda_i/\Lambda'$ of the groups~$H_i$ is \emph{non-degenerate} (which is no surprise). In particular we show that this condition f\/ixes~$\Lambda'$ uniquely. In later applications we often encounter $\hat{f}$ as a natural identif\/ication of $G_1$ and the dual $\hat{G}_2$, e.g., when studying representation theory.

To f\/ind all solutions $f$ with this property we develop a machinery to push $\hat{f}$ into the fundamental group $\pi_1$, which encapsulates all the $\ell$-dependence: In Def\/inition~\ref{def:matrix_A_l} we give an abstract characterization of a~\emph{centralizer transfer map}
\begin{gather*}
	A_\ell\colon \ \Lambda/\Lambda_R \stackrel{\sim}{\longrightarrow} \operatorname{Cent}_{\Lambda}^\ell(\Lambda_R)/\operatorname{Cent}_{\Lambda_R}^\ell(\Lambda)
\end{gather*}
(without \looseness=-2 proving that it always exists). In a generic case this is just multiplication by $\ell$, but it se\-verely depends on common divisors of $\ell$ with root length and divisors of the Cartan matrix. With this matrix we can transfer $q^{-(\mu,\nu)}$ to a natural form $a_g^\ell$ on the fundamental group. We prove that $\hat{f}$ is non-degenerate if\/f $a_g^\ell(\mu,\nu)=q^{-(\mu, A_\ell(\nu))} \cdot g(\mu, A_\ell(\nu))$ is non-degenerate. This explains why the set of solutions, say for fundamental group $\Z_n$ always looks like the subset of invertible elements $\Z_n^\times$ but it is shifted (namely by~$A_\ell$) depending on $\ell$ and the root system in~question.

In Section~\ref{section4} the remaining computational work is done for quasitriagularity: We calculate a list containing $a_g^\ell$ for all simple $\g$, depending on common divisors of~$\ell$ with root length and divisors. We thus write down all solutions for $f$ and hence $R$-matrices. The calculation starts with the Smith normal form for the Cartan matrix in question and uses three cases: For $\Lambda=\Lambda_W$ we have a generic construction, the cases $A_n$ with their large fundamental group $\Z_{n+1}$ is treated by hand, as is $D_{2n}$ with non-cyclic fundamental group, which has the only cases allowing $\Lambda_1\neq \Lambda_2$.

In Section~\ref{section5} we address our main issue of factorizability with our new tools:

\looseness=-1 In Section~\ref{section5.1} we introduce factorizability. Then we calculate the monodromy matrix $R_{21}R$ for an arbitrary choice of $R$-matrix in terms the $R_0$-part. This gives a purely lattice theoretic problem equivalent to the factorizability of such an $R$-matrix. Then we prove in the main Theo\-rem~\ref{thm:InvertyMonodromymat} that factorizability is equivalent to the non-degeneracy of a symmetrization $\operatorname{Sym}_G\big(\hat{f}\big)$ of~$\hat{f}$. As will turn out later, the radical of this form is isomorphic to the group of transparent objects.

In Section~\ref{section5.2} we restrict ourselves to the \emph{symmetric case} where $H_1=H_2$ and $f$, $g$ are symmetric. Other cases appear only in some of
the non-cyclic $\Z_2\times\Z_2$-extension for type $\g=D_{2n}$ and are dealt with in Section~\ref{section5.3} and give surprising new solutions.

The main result for the symmetric case is that the radical of the form $\operatorname{Sym}_G\big(\hat{f}\big)$ is in this case simply the $2$-torsion of~$\Lambda/\Lambda'$ (Example~\ref{ex:RadSymf}) and that this is non-degenerate precisely for odd $\ell$ and odd $\Lambda/\Lambda_R$ as well as for $\g=B_n$, $\Lambda=\Lambda_R$, $\ell\equiv 2$ ${\rm mod}~4$ including~$A_1$.

In Section~\ref{section5.4} we prove the following result:

\begin{lemmaX}
	The transparent objects in the category of representations of the Hopf algebra $u_q(\g,\Lambda)$ with $R$-matrix given by Lusztig's ansatz are $1$-dimensional objects $\C_\chi$ and are the $f$-transformed of the radical of $\operatorname{Sym}_G\big(\hat{f}\big)$:
	\begin{gather*}
		\chi(\mu)=f(\mu,\xi),\qquad \xi\in \operatorname{Rad}\big(\operatorname{Sym}_G\big(\hat{f}\big)\big).
	\end{gather*}
\end{lemmaX}

In the following we summarize our results by a table containing all quasitriangular quantum groups $u_q(\g,\Lambda)$ with  their group of transparent objects. In Section~\ref{section6} we show that all our quasitriangular quantum groups  admit a~ribbon element.  The factorizable solutions and thus modular tensor categories are~$\ell$ odd, $\Lambda=\Lambda_R$ and the following new factorizable cases:

($\ell$ odd, $E_6$, $\Lambda=\Lambda_W$) and ($\ell\equiv 2$ ${\rm mod}~4$, $\g=B_n$, $\Lambda=\Lambda_R$) (including $A_1$) and ($\ell$ odd, $\g=D_{2n}$, $\Lambda_1\neq \Lambda_2$). All other cases can be modularized as discussed in Question~\ref{q:modularize}.

 The columns of the following table are labeled by

\vspace{-3mm}

 \begin{enumerate}\itemsep=-1.5pt
 	\item[1)] the f\/inite-dimensional simple complex Lie algebra~$\g$,
 	\item[2)] the natural number $\ell$, determining the root of unity $q=\exp\big( {\frac{2 \pi i}{\ell}}\big) $,
 	\item[3)] the number of possible $R$-matrices for the Lusztig ansatz,
 	\item[4)] the subgroups $H_i \subseteq H=\Lambda/\Lambda_R$ introduced in Theorem~\ref{thm:solutionsgrpeq},
 	\item[5)] the subgroups $H_i$ in terms of generators given by multiples of fundamental dominant weights $\lambda_i \in \Lambda_W$,
 	\item[6)] the group pairing $g\colon H_1 \times H_2 \to \C^\times$ determined by its values on generators,
 	\item[7)] the group of transparent objects $T \subseteq \Lambda/\Lambda'$ introduced in Lemma~\ref{lm_transparent}.
 \end{enumerate}

\vspace{-4.5mm}

{\footnotesize \centering
\begin{longtable}{@{}c@{\,}||@{\,}c@{\,}|@{\,}c@{\,}||@{\,}c@{\,}|@{\,}c@{\,}|@{\,}c@{\,}|@{\,}c@{}}
$\g$
	&$\ell$
	&\#
	&$H_i\cong$
	&$H_i\,{\scriptstyle (i=1,2)}$
	&$g$
	&$T \subseteq \Lambda/\Lambda'$
	\\ \hline \hline
	\multirow{4}{*}{\text{all}}
	& \multirow{2}{*}{$\ell$ odd}	& \multirow{2}{*}{1}	& \multirow{4}{*}{$\Z_1$} &\multirow{4}{*}{$\car{0}$} &\multirow{4}{*}{$g=1 $} &\multirow{2}{*}{\textbf{0}} \\
	&&&&&&\\ \cline{2-3}\cline{7-7}
	&\multirow{2}{*}{$\ell \equiv 0$ mod $4$ }&\multirow{2}{*}{1}&&&&\multirow{2}{*}{$\Z_2^n$}\\
	&&&&&&\\	
	\hline
	\multirow{6}{*}{} 	& 	& \multirow{6}{*}{$\infty$}	& && &\multirow{2}{*}{$\Z_2^{n-1}$, $2 \nmid x$} 	\\
	&&&\multirow{2}{*}{$\Z_d$} &\multirow{2}{*}{$\car{\hat{d}\lambda_n}$} &\multirow{2}{*}{$g(\hat{d}\lambda_n,\hat{d}\lambda_n)=\exp \big( \frac{2\pi i k}{d}\big) $}&\\
	\multirow{1}{*}{$A_{n\geq 1}$}	&&&&&&\multirow{2}{*}{$\Z_2^{n}$, $2 \,|\, x$}\\
	\multirow{1}{*}{$\pi_1=\Z_{n+1}$}	& 	&	& \multirow{2}{*}{$d\,|\, n+1$}&\multirow{2}{*}{$\hat{d}=\frac{n+1}{d}$}&\multirow{2}{*}{$\operatorname{gcd}\big( d,\frac{k \ell - \hat{d}n}{\operatorname{gcd}(\ell,\hat{d})}\big) =1 $}&\\
	&&&&&&\multirow{2}{*}{$x=\frac{d \ell}{\operatorname{gcd}(\ell,\hat{d})}$}\\
	&&&&&&\\
	\hline
	\multirow{8}{*}{} 	& \multirow{2}{*}{$\ell \equiv 2$ mod $4$}	& \multirow{2}{*}{1}	& \multirow{2}{*}{$\Z_1$} &\multirow{2}{*}{$\car{0}$} &\multirow{2}{*}{$g=1 $} & \multirow{2}{*}{\textbf{0}}	\\
	&&&&&&\\ \cline{2-7}
	& \multirow{2}{*}{$\ell \equiv 2$ mod $4$}& \multirow{2}{*}{2}&\multirow{2}{*}{$\Z_2$}&\multirow{2}{*}{$\car{\lambda_n}$}&\multirow{2}{*}{$g(\lambda_n,\lambda_n)=\pm 1 $}&\multirow{2}{*}{$\Z_2$}\\
	&&&&&&\\ \cline{2-7}
	$B_{n\geq 2}$& \multirow{2}{*}{$\ell \equiv 0$ mod $4$}	& \multirow{2}{*}{2}	& \multirow{2}{*}{$\Z_2$} &\multirow{2}{*}{$\car{\lambda_n}$} &\multirow{2}{*}{$g(\lambda_n,\lambda_n)=\pm 1 $} & \multirow{2}{*}{$\Z^n_2$}	\\
	$\pi_1=\Z_{2}$		&	& 	&&&&\\ \cline{2-7}
	& \multirow{2}{*}{$\ell$ odd} 	& \multirow{2}{*}{1}	&\multirow{2}{*}{$\Z_2$}&\multirow{2}{*}{$\car{\lambda_n}$}&\multirow{2}{*}{$g(\lambda_n,\lambda_n)=(-1)^{n+1}$}&\multirow{2}{*}{$\Z_2$}\\
	&&&&&&\\	
	\hline
	\multirow{8}{*}{} 	& \multirow{2}{*}{$\ell\equiv2\text{ mod }4$}	& \multirow{2}{*}{1}	& \multirow{2}{*}{$\Z_1$}&\multirow{2}{*}{$\car{0}$} &\multirow{2}{*}{$g=1 $} &\multirow{2}{*}{$\Z_2^{n-2}$} 	\\
	&&&&&&\\\cline{2-7}
	&\multirow{2}{*}{$\ell\equiv2\text{ mod }4$}&\multirow{2}{*}{1}&\multirow{6}{*}{$\Z_2$}&\multirow{6}{*}{$\car{\lambda_n}$} &\multirow{2}{*}{$g(\lambda_n,\lambda_n)=1 $} &\multirow{2}{*}{$\Z_2^{n-1}$} \\	
	&&&&&&\\	
	\cline{2-3} \cline{6-7}
	$C_{n\geq 3}$		& \multirow{2}{*}{$\ell\equiv0\text{ mod }4$}	& \multirow{2}{*}{2}	&&&\multirow{2}{*}{$g(\lambda_n,\lambda_n)=\pm 1 $} &\multirow{2}{*}{$\Z_2^n$} \\
	$\pi_1=\Z_{2}$		& 	& 	&&&&	\\ \cline{2-3} \cline{6-7}
	& \multirow{2}{*}{$\ell$ odd}	& \multirow{2}{*}{1}	&&&\multirow{2}{*}{$g(\lambda_n,\lambda_n)=- 1 $}&\multirow{2}{*}{$\Z_2$}\\
	&&&&&&\\	
	\hline
	\multirow{20}{*}{} 	& \multirow{2}{*}{$\ell\equiv2\text{ mod }4$}	& \multirow{2}{*}{1}	& \multirow{2}{*}{$\Z_1$}&\multirow{2}{*}{$\car{0}$} &\multirow{2}{*}{$g=1 $} &\multirow{2}{*}{$\Z_2^{2(n-1)}$} 	\\
	&&&&&&\\\cline{2-7}
	& \multirow{2}{*}{$\ell \equiv 2 \text{ mod }4$}	& \multirow{2}{*}{1}	& \multirow{6}{*}{$\Z_2$} &\multirow{3}{*}{$H_1 \cong\car{\lambda_{2n-1}}$} & \multirow{2}{*}{$g(\lambda_{2n-1},\lambda_{2n})=(-1)^n$} &\multirow{4}{*}{$\Z_2^{2n}$}	\\
	&&&&&&\\	\cline{2-3} \cline{6-6}
	&\multirow{2}{*}{$\ell \equiv 0 \text{ mod }4$}	& 	\multirow{2}{*}{$2\delta_{2 \,|\, n}$}& &&\multirow{2}{*}{$g(\lambda_{2n-1},\lambda_{2n})=\pm 1$, $n$ even} & \\
	& 	&	& &\multirow{3}{*}{$H_2 \cong\car{\lambda_{2n}}$}&&	\\ \cline{2-3} \cline{6-7}
	&\multirow{2}{*}{$\ell$ odd} &\multirow{2}{*}{1}&&&\multirow{2}{*}{$g(\lambda_{2n-1},\lambda_{2n})=-1$}&\multirow{2}{*}{\textbf{0}}\\
	$D_{2n\geq 4}$&&&&&&\\
	\cline{2-7}
	\cline{2-7}
	$\pi_1=\Z_{2} \times \Z_2$&\multirow{2}{*}{$\ell \equiv 2 \text{ mod }4$}&\multirow{2}{*}{$1$}&\multirow{6}{*}{$\Z_2$}&\multirow{6}{*}{$\car{\lambda_{2n}}$}&\multirow{2}{*}{$g(\lambda_{2n},\lambda_{2n})=(-1)^{n+1}$}&\multirow{2}{*}{$\Z_2^{2n-1}$}\\	&&&&&&\\\cline{2-3} \cline{6-7}
	&\multirow{2}{*}{$\ell \equiv 0 \text{ mod }4$}&\multirow{2}{*}{$2\delta_{2 \nmid n}$}&&&\multirow{2}{*}{$g(\lambda_{2n},\lambda_{2n})=\pm 1$, $n$ odd}&\multirow{2}{*}{$\Z_2^{2n}$}\\
	&&&&&&\\\cline{2-3} \cline{6-7}
	&\multirow{2}{*}{$\ell$ odd}&\multirow{2}{*}{1}&&&\multirow{2}{*}{$g(\lambda_{2n},\lambda_{2n})=-1$}&\multirow{2}{*}{$\Z_2$}\\
	&&&&&&\\\cline{2-7}
	&\multirow{2}{*}{$\ell$ even}	&\multirow{2}{*}{$16$}&\multirow{6}{*}{$\Z_2 \times \Z_2$}&\multirow{6}{*}{$\car{\lambda_{2n},\lambda_{2n+1}}$}&\multirow{2}{*}{$g(\lambda_{2(n-1)+i},\lambda_{2(n-1)+j})=\pm 1$}&\multirow{2}{*}{$\Z_{2}^{2n}$}\\
	&&&&&&\\\cline{2-3}\cline{6-7}
	&\multirow{4}{*}{$\ell$ odd}&\multirow{4}{*}{$6$}&&&\multirow{2}{*}{$\det(K)=K_{12}+K_{21}=0$ mod $2$}&\multirow{2}{*}{$\Z_{2}$}\\
	&&&&&&\\\cline{6-7}
	&&&&&\multirow{2}{*}{$\det(K)=K_{12}+K_{21}=1$ mod $2$}&\multirow{2}{*}{$\Z_{2}^2$}\\
	&&&&&&\\
	\hline
\pagebreak
\hline
	\multirow{12}{*}{} 	& \multirow{2}{*}{$\ell\equiv2\text{ mod }4$}	& \multirow{2}{*}{1}	& \multirow{2}{*}{$\Z_1$}&\multirow{2}{*}{$\car{0}$} &\multirow{2}{*}{$g=1 $} &\multirow{2}{*}{$\Z_2^{2n}$} 	\\
	&&&&&&\\\cline{2-7}
	& \multirow{2}{*}{$\ell\equiv2\text{ mod }4$}	& \multirow{2}{*}{1}	& \multirow{6}{*}{$\Z_2$} &\multirow{6}{*}{$\car{2\lambda_{2n+1}}$} &\multirow{2}{*}{$g(2\lambda_{2n+1},2\lambda_{2n+1})=1 $} &\multirow{4}{*}{$\Z^{2n+1}_2$} 	\\
	& 	& 	&&&&\\	\cline{2-3} \cline{6-6}
	& \multirow{2}{*}{$\ell\equiv0\text{ mod }4$}	& \multirow{2}{*}{2}	&&&\multirow{2}{*}{$g(2\lambda_{2n+1},2\lambda_{2n+1})=\pm 1 $} & \\
	& 	& 	&&&&	\\ \cline{2-3} \cline{6-7}
	$D_{2n+1\geq 5}$	& \multirow{2}{*}{$\ell$ odd}	& \multirow{2}{*}{1}	&&&\multirow{2}{*}{$g(2\lambda_{2n+1},2\lambda_{2n+1})=- 1 $}&\multirow{2}{*}{$\Z_2$}\\
	$\pi_1=\Z_{4}$	&&&&&&\\ \cline{2-7}
	&\multirow{2}{*}{$\ell$ even}&\multirow{2}{*}{4}&\multirow{4}{*}{$\Z_4$}&\multirow{4}{*}{$\car{\lambda_{2n+1}}$}&\multirow{2}{*}{$g(\lambda_{2n+1},\lambda_{2n+1})=c$, $c^4=1$}&\multirow{2}{*}{$\Z^{2n+1}_2$}\\
	&&&&&&\\ \cline{2-3}\cline{6-7}
	&\multirow{2}{*}{$\ell$ odd}&\multirow{2}{*}{2}&&&\multirow{2}{*}{$g(\lambda_{2n+1},\lambda_{2n+1})=\pm 1$}&\multirow{2}{*}{$\Z_2$}\\
	&&&&&&\\
	\hline
	\multirow{8}{*}{} 	& \multirow{2}{*}{$\ell\equiv2\text{ mod }4$}	& \multirow{2}{*}{1}	& \multirow{2}{*}{$\Z_1$}&\multirow{2}{*}{$\car{0}$} &\multirow{2}{*}{$g=1 $} &\multirow{2}{*}{$\Z_2^6$} 	\\
	&&&&&&\\\cline{2-7}
	& \multirow{2}{*}{$\ell \equiv 0$ mod $3$}	& \multirow{2}{*}{3}	& \multirow{6}{*}{$\Z_3$} &\multirow{6}{*}{$\car{\lambda_n}$} &\multirow{2}{*}{$g(\lambda_n,\lambda_n)=c$, $c^3=1$} &\multirow{3}{*}{$\Z_2^6$, $2 \,|\, \ell$} 	\\
	&	& 	&&&&\\ \cline{2-3} \cline{6-6}
	$E_{6}$	& \multirow{2}{*}{$\ell \equiv 1$ mod $3$} 	& \multirow{2}{*}{2}	&&&\multirow{2}{*}{$g(\lambda_n,\lambda_n)=1,\exp \big(\frac{2 \pi i 2}{3} \big) $}&\\
	$\pi_1=\Z_{3}$&&&&&&\multirow{3}{*}{\textbf{0}, $2 \nmid \ell$}\\ \cline{2-3} \cline{6-6}
	&\multirow{2}{*}{$\ell \equiv 2$ mod $3$}&\multirow{2}{*}{2}&&&\multirow{2}{*}{$g(\lambda_n,\lambda_n)=1,\exp \big(\frac{2 \pi i }{3} \big) $}&\\	
	&&&&&&\\	
	\hline
	\multirow{6}{*}{} 	& \multirow{2}{*}{$\ell\equiv2\text{ mod }4$}	& \multirow{2}{*}{1}	& \multirow{2}{*}{$\Z_1$}&\multirow{2}{*}{$\car{0}$} &\multirow{2}{*}{$g=1 $} &\multirow{2}{*}{$\Z_2^6$} 	\\
	&&&&&&\\\cline{2-7}	
	& \multirow{2}{*}{$\ell$ even}	& \multirow{2}{*}{2}	& \multirow{4}{*}{$\Z_2$} &\multirow{4}{*}{$\car{\lambda_n}$} &\multirow{2}{*}{$g(\lambda_n,\lambda_n)=\pm 1 $} &\multirow{2}{*}{$\Z_2^n$}\\
	$E_{7}$		&	& 	&&&&\\ \cline{2-3} \cline{6-7}
	$\pi_1=\Z_{2}$	& \multirow{2}{*}{$\ell$ odd} 	& \multirow{2}{*}{1}	&&&\multirow{2}{*}{$g(\lambda_n,\lambda_n)=1$}& \multirow{2}{*}{$\Z_2$}\\
	&&&&&&\\	
	\hline	
	\multirow{2}{*}{$E_{8}$} 	& \multirow{2}{*}{$\ell\equiv2\text{ mod }4$}	& \multirow{2}{*}{1}	& \multirow{2}{*}{$\Z_1$}&\multirow{2}{*}{$\car{0}$} &\multirow{2}{*}{$g=1 $} &\multirow{2}{*}{$\Z_2^8$} 	\\
	&&&&&&\\\cline{2-7}
	\hline	
	\multirow{2}{*}{$F_{4}$} 	& \multirow{2}{*}{$\ell\equiv2\text{ mod }4$}	& \multirow{2}{*}{1}	& \multirow{2}{*}{$\Z_1$}&\multirow{2}{*}{$\car{0}$} &\multirow{2}{*}{$g=1 $} &\multirow{2}{*}{$\Z_2^2$} 	\\
	&&&&&&\\\cline{2-7}
	\hline	
	\multirow{2}{*}{$G_{2}$} 	& \multirow{2}{*}{$\ell\equiv2\text{ mod }4$}	& \multirow{2}{*}{1}	& \multirow{2}{*}{$\Z_1$}&\multirow{2}{*}{$\car{0}$} &\multirow{2}{*}{$g=1 $} &\multirow{2}{*}{$\Z_2^2$} 	\\
	&&&&&&\\\cline{2-7}
	\hline				
	\caption{Solutions for $R_0$-matrices.}
\end{longtable}}
\vspace{-5mm}

In the case $D_{2n}$, $\Lambda=\Lambda_W$, $g$ is uniquely defined by a~$(2\times 2)$-matrix $K \in \mathfrak{gl}(2,\F_2)$, s.t.\ $g(\lambda_{2(n-1)+i},
\lambda_{2(n-1)+j})=\exp \big( \frac{2 \pi i K^g_{ij}}{2} \big)$ for $i,j\in \{1,2\}$.

\vspace{-2mm}

\section{Preliminaries}\label{section2}

\vspace{-1mm}

\subsection{Lie theory}\label{section2.1}

Throughout this article, $\g$ denotes a f\/inite-dimensional simple complex Lie algebra. We f\/ix a~choice of simple roots $\Delta=\{ \alpha_i\,|\,i \in I\}$, so that the Cartan matrix $C$ is given by $C_{ij}=2\frac{(\alpha_i,\alpha_j)}{(\alpha_i,\alpha_i)}$, where $(\,,\,)$ denotes the normalized Killing form. For a root $\alpha$, we def\/ine $d_\alpha:=\frac{(\alpha,\alpha)}{2}$ and set $d_i=d_{\alpha_i}$. By $\Lambda_R:=\mathbb{Z}[\Delta]$ and $\Lambda_R^\vee:=\mathbb{Z}[\Delta^\vee]$ we denote the (co)root lattice of~$\g$.

\looseness=-1 By $\Lambda_W$, we denote the \emph{weight lattice} spanned by fundamental dominant weights $\lambda_i$, which are def\/ined by the equation $(\lambda_i,\alpha_j)=\delta_{i,j}d_i$. Finally, we def\/ine the \emph{co-weight lattice} $\Lambda_W^\vee$ as the $\mathbb{Z}$-span of the elements $\lambda_i^\vee:=\frac{\lambda_i}{d_i}$. The quotient $\pi_1:=\Lambda_W/\Lambda_R$ is called the \emph{fundamental group} of~$\g$.

One can easily see that the Killing form restricts to a perfect pairing $(\,,\,)\colon \Lambda^\vee_W \times \Lambda_R \to \mathbb{Z}$ and that we get a string of inclusions $\Lambda_R \subseteq \Lambda_R^\vee \subseteq \Lambda_W \subseteq \Lambda_W^\vee$.

\vspace{-1.5mm}

\subsection[Lusztig's ansatz for $R$-matrices]{Lusztig's ansatz for $\boldsymbol{R}$-matrices}\label{section2.2}

The starting point for our work \cite{LN14b} was Lusztig's ansatz in \cite[Section~32.1]{Lus93} for a universal $R$-matrix of $U_q(\g)$. Namely, for a specif\/ic element $\bar\Theta\in U_q^{\geq 0}\otimes U_q^{\leq 0}$ from a dual basis and a suitable (not further specif\/ied) element in the coradical $R_0\in U_q^0\otimes U_q^0$ we are looking for $R$-matrices of the form
\begin{gather*}
	R=R_0\bar\Theta.
\end{gather*}
We remark that there is no claim that all possible $R$-matrices are of this form. However they are an interesting source of examples, motivated by the interpretation of $u_q(\g)$ as a quotient of a Drinfeld double and thus well-behaved with respect to the triangular decomposition. This ansatz has been successfully generalized to general diagonal Nichols algebras in~\cite{AY13}. In our more general setting $U_q(\g,\Lambda,\Lambda')$, we have
\begin{gather*}
	R_0\in \C[\Lambda/\Lambda']\otimes\C[\Lambda/\Lambda'].
\end{gather*}

This ansatz has been worked out by M\"uller in his dissertation \cite{Mue98a, Mue98b,Mue98b+} for small quantum groups $u_q(\g)$ which we will use in the following, leading to a system of quadratic equation on~$R_0$ that are equivalent to $R$ being an $R$-matrix:

\begin{Theorem}[{cf.\ \cite[Theorem~8.2]{Mue98b+}}]\label{thm:Theta}
Let $u:=u_q(\g)$.
\begin{enumerate}\itemsep=0pt
\item[$(a)$] There is a unique family of elements $\Theta_\beta\in u_{\beta}^-\otimes u_{\beta}^+$, $\beta\in\Lambda_R$, such that $\Theta_0=1\otimes 1$ and
 $\Theta=\sum_{\beta}\Theta_{\beta}\in u\otimes u$ satisfies $\Delta(x)\Theta = \Theta \bar\Delta(x)$ for all $x\in\ u$.
\item[$(b)$] Let $B$ be a vector space-basis of $u^-$, such that $B_{\beta}=B\cap u^-_{\beta}$ is a basis of~$u^-_{\beta}$ for all $\beta$. Here, $u_{\beta}^-$ refers to the natural $\Lambda_R$-grading of $u^-$. Let $\{b^* \,|\, b\in B_{\beta}\}$ be the basis of $u_{\beta}^+$ dual to $B_{\beta}$ under the non-degenerate bilinear form $(\,\cdot\,,\,\cdot\,)\colon u^-\otimes u^+\to \C$. We have
 \begin{gather*}
 \Theta_{\beta} = (-1)^{{\rm tr} \beta} q_{\beta} \sum_{b\in B_{\beta}} b^-\otimes b^{*+} \in u_{\beta}^- \otimes u_{\beta}^+.
 \end{gather*}
\end{enumerate}
\end{Theorem}

\begin{Theorem}[{cf.\ \cite[Theorem 8.11]{Mue98b+}}]\label{thm:R0}
Let $\Lambda'\subset\{\mu\in\Lambda\,|\, K_{\mu} \text{ central in } u_q(\g,\Lambda)\}$ a subgroup of~$\Lambda$, and $G_1$, $G_2$ subgroups of $G:=\Lambda/\Lambda'$, containing $\Lambda_R/\Lambda'$. In the following, $\mu,\mu_1,\mu_2\in G_1$ and $\nu,\nu_1,\nu_2\in G_2$.

The element $R=R_0\bar\Theta$ with an arbitrary $R_0= \!\sum\limits_{\mu,\nu} f(\mu,\nu) K_{\mu} \otimes K_{\nu}$ is a $R$-matrix for $u_q(\g,\Lambda,\Lambda')$, if and only if for all $\alpha\in \Lambda_R$ and $\mu$, $\nu$ the following holds:
 \begin{gather}
 f(\mu+ \alpha, \nu) = q^{-(\nu, \alpha)} f(\mu,\nu), \qquad f(\mu, \nu+ \alpha) = q^{-(\mu, \alpha)} f(\mu,\nu), \label{f01} \\
 \sum_{\nu_1+\nu_2 = \nu} f(\mu_1,\nu_1)f(\mu_2,\nu_2) = \delta_{\mu_1,\mu_2} f(\mu_1,\nu),\qquad
 \sum_{\mu_1+\mu_2 = \mu} f(\mu_1,\nu_1)f(\mu_2,\nu_2) = \delta_{\nu_1,\nu_2} f(\mu,\nu_1),\nonumber\\ 
 \sum_{\mu} f(\mu,\nu) = \delta_{\nu,0},\qquad \sum_{\nu} f(\mu,\nu) = \delta_{\mu,0}.\nonumber 
 \end{gather}
\end{Theorem}

\section[Conditions for the existence of $R$-matrices]{Conditions for the existence of $\boldsymbol{R}$-matrices}\label{section3}

\subsection[A f\/irst set of conditions on $\Lambda/\Lambda'$]{A f\/irst set of conditions on $\boldsymbol{\Lambda/\Lambda'}$}\label{section3.1}
The target of our ef\/forts is a Hopf algebra called small quantum group $u_q(\g,\Lambda,\Lambda')$ with Cartan part $u_q^0=\C[\Lambda/\Lambda']$. It is def\/ined, e.g., in~\cite{LN14b} and depends on lattices $\Lambda$, $\Lambda'$ def\/ined below. For $\Lambda=\Lambda_R$ the root lattice and this is the usual small quantum group; the choice of $\Lambda'$ dif\/fers in literature.

In the previous section we have discussed an $R=R_0\bar\Theta$-matrix for the quantum group $u_q(\g,\Lambda,\Lambda')$ can be obtained from an $R_0$-matrix of the form
\begin{gather*}
	R_0= \sum_{\mu,\nu\in \Lambda} f(\mu,\nu) K_{\mu} \otimes K_{\nu}\in \C[\Lambda/\Lambda']\otimes\C[\Lambda/\Lambda'].
\end{gather*}
In the following we collect necessary and suf\/f\/icient conditions for $R=R_0\bar{\Theta}$ to be an $R$-matrix.

\begin{Definition}We f\/ix once-and-for-all a f\/inite-dimensional simple complex Lie algebra $\g$ and a lattice $\Lambda$ between root- and weight-lattice
\begin{gather*}
\Lambda_R\subseteq \Lambda \subseteq \Lambda_W.
\end{gather*}
These choices have a nice geometric interpretation as quantum groups associated	to dif\/ferent Lie groups associated to the Lie algebra $\g$.
\end{Definition}
Another interesting choice is $\Lambda_R \subseteq \Lambda \subseteq \Lambda_W^\vee \cong \Lambda_R^*$, which would below pose no additional complications and may produce further interesting factorizable $R$-matrices.

\begin{Definition} We f\/ix once-and-for-all a primitive $\ell$-th root of unity $q$. For $\Lambda_1,\Lambda_2 \subseteq \Lambda_W^\vee$ we def\/ine the sublattice
	\begin{gather*}
		\operatorname{Cent}_{\Lambda_1}(\Lambda_2):= \{\,\nu \in \Lambda_1 \, | \, (\nu,\mu) \in \ell \cdot \mathbb{Z} \ \forall\, \mu \in \Lambda_2 \}.
	\end{gather*}
	Informally, this is the centralizer with respect to the braiding $q^{-(\nu,\mu)}$.
\end{Definition}

Contrary to \cite{LN14b} we do not f\/ix $\Lambda'$ but we prove later Corollary~\ref{cor:LambdaPrime} that there is a necessary choice for~$\Lambda'$. In this way, we get more solutions than in~\cite{LN14b}. The only condition necessary to ensure that the Hopf algebra $u_q(\g,\Lambda,\Lambda')$ is well-def\/ined is
$\Lambda'\subseteq \operatorname{Cent}_{\Lambda_R}(\Lambda_R)$.

\begin{Theorem}[{cf.\ \cite[Theorem~3.4]{LN14b}}]\label{thm:solutionsgrpeq}
 The $R_0$-matrix is necessarily of the form
 \begin{gather*}
 f(\mu,\nu)=\frac{1}{d|\Lambda_R/\Lambda'|}\cdot q^{-(\mu, \nu)}g(\bar{\mu},\bar{\nu})\delta_{\bar{\mu}\in
 	H_1}\delta_{\bar{\nu}\in H_2},
 \end{gather*}
 where $H_1$, $H_2$ are subgroups of $H:=\Lambda/\Lambda_R \subseteq \pi_1$ with equal cardinality $|H_1|=|H_2|=:d$ $($not necessarily isomorphic!$)$ and $g\colon H_1\times H_2\to\C^{\times}$ is a pairing of groups.
\end{Theorem}
The necessity of this form (in particular that the support of $f$ is indeed a subgroup!) amounts to a combinatorial problem of its own interest, which we solved for $\pi_1$ cyclic in \cite{LN14a} and for $\Z_2\times \Z_2$ by hand; a closed proof for all abelian groups would be interesting.

\begin{Definition}	Let $g\colon G \times H \to \mathbb{C}^\times$ be a f\/inite group pairing, then the \emph{left radical} is def\/ined as
	\begin{gather*}
		\operatorname{Rad}_L(g):=\{ \lambda \in G \,|\, g(\lambda,\eta)=1 \, \forall \, \eta \in H \}.
	\end{gather*}
	Similarly, the right radical is def\/ined as
	\begin{gather*}
		\operatorname{Rad}_R(g):=\{ \eta \in H \,|\, g(\lambda,\eta)=1 \, \forall \, \lambda \in G \}.
	\end{gather*}
	The pairing $g$ is called \emph{non-degenerate} if $\operatorname{Rad}_L(g)=0$. If in addition $\operatorname{Rad}_R(g)=0$, $g$ is called \emph{perfect}.
\end{Definition}

For an $R_0$-matrix of this form, a suf\/f\/icient condition is that they fulf\/ill the so-called \emph{diamond-equations} (see \cite[Def\/inition~2.7]{LN14b}) for each element $0\neq \zeta\in (\operatorname{Cent}(\Lambda_R)\cap\Lambda)/\Lambda'$.

However, we will now go into a dif\/ferent, more systematic direction that makes use of the following observation:
\begin{Lemma}\label{lm:NondegGroupPairing}
An $R_0$-matrix of the form given in Theorem~{\rm \ref{thm:solutionsgrpeq}} is a solution to the equations in Theorem~{\rm \ref{thm:R0}}, and hence produces an $R$-matrix $R_0\bar{\Theta}$ iff the restriction to the support
	\begin{gather*}
		\hat{f}:=d |\Lambda_R/\Lambda'|\cdot f\colon \ G_1 \times G_2 \to \C^\times
	\end{gather*}
	is a \emph{perfect} group pairing, where $G_i:=\Lambda_i/\Lambda' \subseteq \Lambda/\Lambda'=:G$.
\end{Lemma}
\begin{proof}
	We f\/irst show that a solution with restriction to the support a nondegenerate pairing solves the equation.
	
The f\/irst equations are obviously fulf\/illed for the form assumed
\begin{gather*}
	f(\mu+ \alpha, \nu) = q^{-(\nu, \alpha)} f(\mu,\nu),\qquad f(\mu, \nu+ \alpha) = q^{-(\mu, \alpha)} f(\mu,\nu).
\end{gather*}
For the other equations the sums get only contributions in the support $\Lambda_1/\Lambda' \times \Lambda_2/\Lambda'$. The quantities $f(\mu,\nu)\cdot d|\Lambda_R/\Lambda'|$ for f\/ixed~$\nu$ (or~$\mu$) are characters on the respective support, and by the assumed non-degeneracy all $\nu\neq 0$ give rise to dif\/ferent nontrivial characters. Then the second and third relations follows from orthogonality of characters. Note that since \smash{$d|\Lambda_R/\Lambda'|=|G_1|=|G_2|$} (equality of the latter was an assumption!) we were able to chose the right normalization.
	
For the other direction assume a solution of the given form to the equations. Then already the third equation shows that no $f(-,\nu)$ may be the trivial character and hence the form on the support is nondegenerate and hence perfect by $|G_1|=|G_2|$.
\end{proof}	

\begin{Corollary}\label{cor:LambdaPrime}
A first consequence of the perfectness of $\hat{f}$ $($i.e., a necessary condition for quasi-triangularity$)$ is
\begin{gather*}
	\operatorname{Cent}_{\Lambda_R}(\Lambda_1)=\operatorname{Cent}_{\Lambda_R}(\Lambda_2)=\Lambda'.
\end{gather*}
This fixes $\Lambda'$ uniquely. Moreover in cases $\Lambda_1\neq \Lambda_2$, which can only happen for~$\g=D_{2n}$, where $\pi_1$ is noncyclic, we get an additional constraint relating $\Lambda_1$, $\Lambda_2$.
\end{Corollary}

In our case, the only possibility for $\Lambda_1 \neq \Lambda_2$, s.t.\ $G_1 \cong G_2$ is $\g =D_{2n}$. In this case, we have $\operatorname{Cent}_{\Lambda_R}(\Lambda_W)=\operatorname{Cent}_{\Lambda_R}(\Lambda_R)$ and thus the above condition is always fulf\/illed.

Our main goal for the new approach on quasitriangularity as well as the later modularity is to reduce this non-degeneracy condition for $\hat{f}$ to a non-degeneracy condition for $g$ on $H_1,H_2\subset \pi_1$ that can be checked explicitly.

\subsection{A natural form on the fundamental group}\label{section3.2}

We now def\/ine for each triple $(\Lambda,\Lambda_1,\Lambda_2)$ and each $\ell$th root of unity $q$ a natural pairing $a_\ell$ on the subgroups $H_i:=\Lambda_i/\Lambda_R$ of the fundamental group $\pi_1:=\Lambda_W/\Lambda_R$. The simplest example is $a_\ell=e^{-2\pi i(\mu,\nu)}$. In general it is a transportation of the natural form $q^{-(\mu,\nu)}$ (which does not factorize over $\Lambda_R$) to $H_i$ by a suitable isomorphism $A_\ell$.

This isomorphism $A_\ell$ will encapsulate the crucial dependence on the common divisors of~$\ell$,~$|H|$ and the root lengths $d_i$; moreover, for dif\/ferent $H$ these forms are \emph{not} simply restrictions of one another.

Then, we can moreover transport any given pairing $g$ together with $q^{-(\mu,\nu)}$ along the isomorphism $A_\ell$ to the $H_i$ and thus def\/ine forms $a_\ell^g$ on $H$. The main result of this section is in Theorem~\ref{cor:factorizableAell} that the non-degeneracy condition in Lemma~\ref{lm:NondegGroupPairing} for~$R_0(f)$ depending on~$H_i$, $g$~is equivalent to $a_\ell^g$ being non-degenerate.

\begin{Definition}Let $\Lambda \subseteq \Lambda_W^\vee$ be a sublattice, s.t.\ $\Lambda_R \subseteq \Lambda$. By $\hat{\Lambda} \subset \Lambda_W^\vee$ we denote the unique sub\-lattice, s.t.\ the symmetric bilinear form $(\,\cdot\,,\,\cdot\,)\colon \Lambda_W^\vee \times \Lambda_W^\vee \to \mathbb{Q}$ induces a commuting diagram{\samepage
\begin{equation*}
	\begin{tikzcd}
	\Lambda_R \arrow[hookrightarrow]{r}\arrow{d}{\cong} & \hat{\Lambda} \arrow[hookrightarrow]{r}\arrow{d}{\cong} & \Lambda_W^\vee \arrow{d}{\cong} \\
	{\Lambda_W^\vee}^* \arrow[hookrightarrow]{r} & \Lambda^* \arrow[hookrightarrow]{r} & \Lambda_R^* ,
	\end{tikzcd}
	\end{equation*}
	where $\Lambda^*:= \operatorname{Hom}_\mathbb{Z}(\Lambda,\mathbb{Z})$. In particular, we have $\hat{\Lambda}_R=\Lambda_W^\vee$ and $\hat{\Lambda}_W^\vee=\Lambda_R$.}
\end{Definition}

\begin{Definition} \label{def:matrix_A_l}
	A \emph{centralizer transfer map} is an group endomorphism $A_\ell \in \operatorname{End}_{\mathbb{Z}}(\Lambda)$, s.t.\
	\begin{enumerate}\itemsep=0pt
		\item[1)] $A_\ell(\Lambda) \stackrel{!}{=} \Lambda \cap \ell \cdot \hat{\Lambda}_R=\operatorname{Cent}_{\Lambda}^\ell(\Lambda_R)$,
		\item[2)] $A_\ell(\Lambda_R) \stackrel{!}{=} \Lambda_R \cap \ell \cdot \hat{\Lambda}=\operatorname{Cent}_{\Lambda_R}^\ell(\Lambda)$.
	\end{enumerate}
Such a $A_\ell$ induces a group isomorphism
	\begin{gather*}
	\Lambda/\Lambda_R \stackrel{\sim}{\longrightarrow} \operatorname{Cent}_{\Lambda}^\ell(\Lambda_R)/\operatorname{Cent}_{\Lambda_R}^\ell(\Lambda).
	\end{gather*}
	Of course $A_\ell$ is not unique.
\end{Definition}
\begin{Question}
	Are there abstract arguments for the existence of these isomorphism and for its explicit form?
\end{Question}
We will calculate explicit expressions for $A_\ell$ depending on the cases in the next section. At this point we give the generic answers:
\begin{Example}
	For $\Lambda=\Lambda_W^\vee$ we have $A_\ell=\ell\cdot\operatorname{id}$.

	For $\Lambda=\Lambda_R$ the two conditions are equivalent, so existence is trivial (resp.\ obviously the two trivial groups are isomorphic) and we may simply take for $A_\ell$ any base change between left and right side. The expression may however be nontrivial.
\end{Example}
\begin{Lemma} \label{lm:AellForGCD=1}
	Assume $\gcd(\ell, |\Lambda_W^\vee/\Lambda|)=1$, then $A_\ell=\ell\cdot \operatorname{id}$. In particular this is the case if $\ell$ is prime to all root lengths and all divisors of the Cartan matrix.
	
Moreover if $\ell=\ell_1\ell_2$ with $\gcd(\ell_1, |\Lambda_W^\vee/\Lambda|)=1$, then $A_\ell=\ell_1\cdot A_{\ell_2}$.
	
This means we only have to calculate $A_\ell$ for all divisors $\ell$ of $|\Lambda_W^\vee/\Lambda|$, which is a subset of all divisors of root lengths times divisors of the Cartan matrix.
\end{Lemma}
\begin{proof}For the f\/irst condition we need to show for any $\lambda\in \Lambda_W^\vee$ that $\ell\lambda\in\Lambda$ already implies $\lambda\in\Lambda$. But if by assumption the order of the quotient group $\Lambda_W^\vee/\Lambda$ is prime to $\ell$, then $\ell\cdot$ is an isomorphism on this abelian group, hence follows the assertion. For the second condition applies the same argument noting that $|\hat{\Lambda}/\Lambda_R|=|\Lambda_W^\vee/\Lambda|$.
	
For the second claim we simply consider the inclusion chains
\begin{gather*}
A_\ell(\Lambda)	\subset \Lambda \cap \ell_2 \cdot \hat{\Lambda}_R\subset \Lambda \cap \ell \cdot \hat{\Lambda}_R,\\
A_\ell(\Lambda_R)\subset \Lambda \cap \ell_2 \cdot \hat{\Lambda}\subset \Lambda_R \cap \ell \cdot \hat{\Lambda},
\end{gather*}
where a f\/irst isomorphism is given by $A_{\ell_2}$ and again $\ell_1\cdot$ is a second isomorphism because it is prime to the index.
\end{proof}	
	
Our main result of this chapter is the following:

\begin{Theorem} \label{thm:radical} Let $\Lambda_R \subseteq \Lambda_1$, $\Lambda_2 \subseteq \Lambda_W$ be intermediate lattices, s.t.\ the condition in Corollary~{\rm \ref{cor:LambdaPrime}} is fulfilled, i.e., $\operatorname{Cent}_{\Lambda_R}(\Lambda_1)=\operatorname{Cent}_{\Lambda_R}(\Lambda_2)=\Lambda'$. Assume we have a centralizer transfer map~$A_\ell$.
	\begin{enumerate}\itemsep=0pt
		\item[$1.$] The following form is well defined on the quotients:
		\begin{align*}
		a_g^\ell\colon \ & \Lambda_1/\Lambda_R \times \Lambda_2/\Lambda_R \longrightarrow \mathbb{C}^\times, \\
		& (\bar{\lambda},\bar{\mu}) \longmapsto q^{-(\lambda,A_\ell(\mu))} \cdot g(\lambda,A_\ell(\mu)).
		\end{align*}
		\item[$2.$] Let
\begin{gather*}
\operatorname{Cent}_{\Lambda_1}^g(\Lambda_2):=\big\{\lambda \in \Lambda_1 \,|\, q^{(\lambda,\mu)}=g(\lambda,\mu) \; \forall\, \mu \in \Lambda_2 \big\}.
\end{gather*}
Then the inclusion $\operatorname{Cent}_{\Lambda_1}^g({\Lambda_2}) \hookrightarrow \Lambda_1$ induces an isomorphism
\begin{gather*}
	\operatorname{Cent}_{\Lambda_1}^g({\Lambda_2})/\Lambda' \cong \operatorname{Rad}\big(a_g^\ell\big).
\end{gather*}
\end{enumerate}
\end{Theorem}

\begin{Corollary}\label{cor:factorizableAell}
The quasitriangularity conditions for a choice $R_0$ are by Lemma~{\rm \ref{lm:NondegGroupPairing}} equivalent to the non-degeneracy of the group pairing on $\Lambda_1/\Lambda' \times \Lambda_2/\Lambda'$:
	\begin{gather*}
		\hat{f}(\lambda,\mu)=q^{-(\lambda,\mu)} g(\lambda,\mu).
	\end{gather*}
By the previous theorem this condition is now equivalent to the nondegeneracy of~$ a_g^\ell$.
\end{Corollary}

This condition on the fundamental group, which is a f\/inite abelian group and mostly cyclic, can be checked explicitly once $a_g^\ell$ has been calculated.

\begin{proof}[Proof of Theorem~\ref{thm:radical}]	The f\/irst part of the theorem is a direct consequence of the def\/inition of the centralizer transfer matrix~$A_\ell$. For the second part, we f\/irst notice that by assumption we have a commutative diagram of f\/inite abelian groups
	\begin{equation*}
		\begin{tikzcd}
		\Lambda_R/\Lambda' \arrow[hookrightarrow]{r}\arrow{d}{q^{-(\cdot,\cdot)}} &\Lambda_1/\Lambda' \arrow[twoheadrightarrow]{r}\arrow{d}{\hat{f}} &\Lambda_1/\Lambda_R \arrow{d}{\hat{f}'}\\
		\left(\Lambda_2/\operatorname{Cent}_{\Lambda_2}(\Lambda_R) \right)^\wedge \arrow[hookrightarrow]{r}& (\Lambda_2/\Lambda')^\wedge \arrow[twoheadrightarrow]{r}& \left(\operatorname{Cent}_{\Lambda_2}(\Lambda_R) /\Lambda' \right)^\wedge,
		\end{tikzcd}
	\end{equation*}
where $G^\wedge$ denotes the dual group of a group $G$.

Now, by the f\/ive lemma we know that $\hat{f}$ is an isomorphism if and only if the induced map~$\hat{f}'$ is an isomorphism. Post-composing this map with the dualized centralizer transfer matrix $A_\ell^\wedge\colon (\operatorname{Cent}_{\Lambda_2}(\Lambda_R) /\Lambda' )^\wedge \cong (\Lambda_2/\Lambda_R)^\wedge$ gives~$a_g^\ell$.
\end{proof}

\section[Explicit calculation for every $\g$]{Explicit calculation for every $\boldsymbol{\g}$}\label{section4}
In the following, we want to compute the endomorphism $A_\ell \in \operatorname{End}_\mathbb{Z}(\Lambda)$ and the pairing~$a_\ell$ on the fundamental group explicitly in terms of the Cartan matrices and the common divisors of $\ell$ with root lengths and divisors of the Cartan matrix. We will f\/inally give a list for all~$\g$.

\subsection{Technical tools}\label{section4.1} We choose the basis of simple roots $\alpha_i$ for $\Lambda_R$ and the dual basis of fundamental coweights $\lambda_i^\vee$ for the dual lattice $\Lambda_W^\vee$ with $(\alpha_i,\lambda_j^\vee)=\delta_{i,j}$.

For any choice $\Lambda\subset \Lambda_W\subset \Lambda_W^\vee$, let $A_\Lambda$ be a \emph{basis matrix}, i.e., any $\Z$-linear isomorphism $\Lambda_W^\vee\to \Lambda$ sending the basis $\lambda_i^\vee$ of $\Lambda_W^\vee$ to some basis $\mu_i$ of $\Lambda$. It is unique up to pre-composition of a unimodular matrix $U \in \mathrm{SL}_n(\mathbb{Z})$.

The dual basis $A_{\hat{\Lambda}}$ of $\hat{\Lambda}$ is def\/ined by
\begin{gather*}
	\big(A_{\hat{\Lambda}}\big(\lambda_i^\vee\big),A_{\Lambda}\big(\lambda_j^\vee\big)\big)= \delta_{ij}.
\end{gather*}
Explicitly, $A_{\hat{\Lambda}}$ is given by $ A_{\hat{\Lambda}} = \big(A_{\Lambda}^{-1}A_R\big)^T,$ where $(A_R)_{ij}=(\alpha_i,\alpha_j)$. Now, let $A_\Lambda=P_\Lambda S_\Lambda Q_\Lambda$ be the unique Smith decomposition of $A_\Lambda$, which means: $P_\Lambda$, $Q_\Lambda$ are unimodular and $S_\Lambda$ is diagonal with diagonal entries $(S_\Lambda)_{ii}=:d^\Lambda_i$, such that $d^\Lambda_i \,|\, d^\Lambda_j$ for $i<j$.
\begin{Example}\label{ex:ExplicitFormOfSNF}
	For the root lattice the $d^{\Lambda_R}_i$ are the divisors of scalar product matrix $(\alpha_i,\alpha_j)$. Their product is
\begin{gather*}
	\prod_i d^{\Lambda_R}_i=\big|\Lambda_W^\vee/\Lambda_R\big|= \Big( \prod_i d_i\Big) \cdot |\pi_1|, \qquad d_i=\frac{(\alpha_i,\alpha_i)}{2}.
\end{gather*}	
	For the coweight lattice all $d^{\Lambda_{W}^\vee}_i=1$. For the weight lattice we recover the familiar $d^{\Lambda_{W}}_i=d_i$.
\end{Example}
Without loss of generality, we will assume the basis matrices $A_\Lambda$ to be symmetric, i.e., \smash{$Q_\Lambda = P_\Lambda^T$}. We then have the following lemma:
\begin{Lemma}\label{lm:ExplicitFormOfCentralizers} Let $\Lambda_R \subseteq \Lambda \subseteq \Lambda_W^\vee$ be a lattice. We define lattices
\begin{gather*}
A_{\rm Cent}:= \big(P_\Lambda^T\big)^{-1} D_\ell P_{\Lambda}^{-1}, \qquad D_\ell:= \operatorname{Diag}\left( \frac{\ell}{\operatorname{gcd}\big(\ell,d_i^\Lambda\big)}\right).
\end{gather*}
Then,
\begin{gather*}
\operatorname{Cent}_{\Lambda_R}(\Lambda)=A_RA_{\rm Cent}\Lambda_W^\vee, \qquad \operatorname{Cent}_{\Lambda}(\Lambda_R)=A_\Lambda A_{\rm Cent}\Lambda_W^\vee.
\end{gather*}
\end{Lemma}

\begin{proof}
We compute explicitly,
\begin{gather*}
\operatorname{Cent}_{\Lambda_R}(\Lambda) = \Lambda_R \cap \ell \cdot \hat{\Lambda}=A_R \Lambda_W^\vee \cap \big(A_{\Lambda}^{-1}A_R\big)^T \ell \Lambda_W^\vee\\
\hphantom{\operatorname{Cent}_{\Lambda_R}(\Lambda)}{} =\big(A_{\Lambda}^{-1}A_R\big)^T \big(\big(\big(A_{\Lambda}^{-1}A_R\big)^T\big)^{-1}A_R \Lambda_W^\vee \cap \ell \Lambda_W^\vee\big)=A_R A_{\Lambda}^{-1} \big(A_\Lambda \cap \ell \Lambda_W^\vee\big)\\
\hphantom{\operatorname{Cent}_{\Lambda_R}(\Lambda)}{}
	=A_R \big(P_\Lambda S_\Lambda P_\Lambda^T\big)^{-1} \big(P_\Lambda S_\Lambda P_\Lambda^T \Lambda_W^\vee\cap \ell \Lambda_W^\vee\big)
	=A_R \big(P_\Lambda^T\big)^{-1} S_\Lambda^{-1} \big(S_\Lambda \Lambda_W^\vee\cap \ell \Lambda_W^\vee\big) \\
\hphantom{\operatorname{Cent}_{\Lambda_R}(\Lambda)}{}
	=A_R \big(P_\Lambda^T\big)^{-1} S_\Lambda^{-1} \operatorname{Diag}(\operatorname{lcm}(S_{\Lambda_{ii}},\ell))\Lambda_W^\vee
	=A_R \big(P_\Lambda^T\big)^{-1} D_\ell \Lambda_W^\vee=A_RA_{\rm Cent}\Lambda_W^\vee.
	\end{gather*}
	On the other hand,
\begin{gather*}
\operatorname{Cent}_{\Lambda}(\Lambda_R) = \Lambda \cup \ell \hat{\Lambda}_R = \Lambda \cup \ell \Lambda_W^\vee = A_\Lambda \Lambda_W^\vee \cup \ell \Lambda_W^\vee 	= P_\Lambda S_\Lambda P_\Lambda^T \Lambda_W^\vee \cup \ell \Lambda_W^\vee \\
\hphantom{\operatorname{Cent}_{\Lambda}(\Lambda_R)}{}
	= P_\Lambda \big(S_\Lambda \Lambda_W^\vee \cup \ell \Lambda_W^\vee\big) = P_\Lambda S_\Lambda D_\ell \Lambda_W^\vee
	= A_\Lambda \big(P_\Lambda^T\big)^{-1} D_\ell \Lambda_W^\vee=A_\Lambda A_{\rm Cent}\Lambda_W^\vee.			
	\end{gather*}
	In particular, this means $A_{\hat{\Lambda}}\operatorname{Cent}_{\Lambda}(\Lambda_R)=\operatorname{Cent}_{\Lambda_R}(\Lambda)$.
\end{proof}
\subsection[Case $\Lambda=\Lambda_W$]{Case $\boldsymbol{\Lambda=\Lambda_W}$}\label{section4.2}

In order to exhaust all cases that appear in our setting, we continue with $\Lambda=\Lambda_W$:

\begin{Lemma}\label{lm:AellForLambda=LambdaW}
In the case $\Lambda=\Lambda_W$, the centralizer transfer matrix $A_\ell$ is of the following form:
\begin{gather*}
A_\ell =	\begin{cases}
A_{\Lambda_W}A_{\rm Cent} Q_C^TP_C^{-1}A_{\Lambda_W}^{-1}, &\operatorname{gcd}(\ell,|\pi_1|) \neq 1,\\
\ell \cdot \operatorname{id}, & \text{else}.
\end{cases}
\end{gather*}
Here, $C=P_CS_CQ_C$ denotes the Smith decomposition of the Cartan matrix of $\g$.
\end{Lemma}

\begin{proof}As we noted in Example \ref{ex:ExplicitFormOfSNF}, we have $A_{\Lambda_W}=\operatorname{Diag}(d_i)$, for $d_i$ being the $i$th root length. Since $d_i \in \{1,p\}$ for some prime number $p$, up to a permutation $A_{\Lambda_W}$ is already in Smith normal form: this means that $P_{\Lambda_W}$ is a permutation matrix of the form $(P_{\Lambda_W})_{ij}=\delta_{j,\sigma(i)}$ for some $\sigma \in S_n$, s.t.\ $d_{\sigma(1)}\leq \cdots \leq d_{\sigma(n)}$. It follows that $A_{\rm Cent}=\operatorname{Diag}\big( \frac{\ell}{\gcd(\ell,d_i)}\big)$.

Using the def\/inition $C_{ij}=\frac{(\alpha_i,\alpha_j)}{d_i}$, in the case $\operatorname{gcd}(\ell,|\pi_1|) \neq 1$ we obtain
\begin{gather*}
A_{\rm Cent}C^T=CA_{\rm Cent}.
\end{gather*}
	Thus,
\begin{gather*}
A_\ell A_R= A_{\Lambda_W}A_{\rm Cent} Q_C^TP_C^{-1}A_{\Lambda_W}^{-1}A_R =A_RC^{-1}A_{\rm Cent}Q_C^TP_C^{-1}C \\
\hphantom{A_\ell A_R}{} =A_R A_{\rm Cent} \big(C^T\big)^{-1}Q_C^TP_C^{-1}C=A_R A_{\rm Cent}.
\end{gather*}
By the previous lemma, this proves the f\/irst condition for~$A_\ell$. The second condition follows immediately from the previous lemma.

The case $\operatorname{gcd}(\ell,|\pi_1|) = 1$ follows from Lemma~\ref{lm:AellForGCD=1} and the fact that $|\pi_1|=|\Lambda_W^\vee/\Lambda_R^\vee|$.
\end{proof}

\subsection[Case $A_n$]{Case $\boldsymbol{A_n}$}\label{section4.3}
In the following example, we treat the case $\g=A_n$ with fundamental group $\Lambda_W/\Lambda_R=\Z_{n+1}$ for general intermediate lattices $\Lambda_R \subseteq \Lambda \subseteq \Lambda_W$.
\begin{Example}\label{ex:AellForAn}In order to compute the centralizer transfer map $A_\ell$, we f\/irst compute the Smith decomposition of $A_R$:
	\begin{gather*}
	A_R=
	\begin{pmatrix}
	2	& -1& 0	& \dots	& 	& 	 0 \\
	-1	& 2 &-1	& 0 	& 	& 	 0 \\
	0	& -1& 2 & \ddots & \ddots & \vdots \\
	0	&0 & \ddots & \ddots & -1 & 0 \\
	\vdots	& & \ddots & -1 & 2 & -1 \\
	0	& \dots & & 0 & -1 & 2
	\end{pmatrix}\\
\hphantom{A_R}{}	=
	\arraycolsep=4pt\def\arraystretch{0.8}
	\begin{pmatrix}
	-1	& 0& 0	& \dots	& 	& 	 0 \\
	2	& -1 &0	& 	& 	& 	 0 \\
	0	& 2& -1 & \ddots & & \vdots \\
	0	&0 & \ddots & \ddots & & 0 \\
	\vdots	& & \ddots & 2 & -1 & 0 \\
	0	& \dots & & 0 & 2 & 1
	\end{pmatrix}\!
	\begin{pmatrix}
	1	& 0& 0	& \dots	& 	& 	 0 \\
	0	& 1 &0	& 	& 	& 	 0 \\
	0	& 0& 1 & \ddots & & \vdots \\
	\vdots& & \ddots & \ddots & & \\
	& & \ddots & & 1 & 0 \\
	0	& \dots & & & 0 & n+1
	\end{pmatrix}\!
	\begin{pmatrix}
	-2	& 1& 0	& \dots	& 	& 	 0 \\
	-3	& 0 &1	& \ddots 	& 	& 	 0 \\
	-4	& 0& 0 & \ddots & & \vdots \\
	\vdots& \vdots & & \ddots & 1 & 0 \\
	-n & & & & 0 & 1 \\
	1	& 0 & \dots & & 0 & 0
	\end{pmatrix}.
	\end{gather*}
A sublattice $\Lambda_R \subsetneq \Lambda \subsetneq \Lambda_W$ is uniquely determined by a divisor $d \,|\, n+1$, so that $\Lambda/\Lambda_R \cong \Z_d$ and is generated by the multiple $\hat{d}\lambda_n$, where $\hat{d}:=\frac{n+1}{d}$. Then
	\begin{gather*} d_i^\Lambda =
	\begin{cases}
	1,& i<n,\\
	d, & i=n.
	\end{cases}
	\end{gather*}
Since $A_n$ is simply laced with cyclic fundamental group, the formula $A_\Lambda=P_R S_\Lambda P_R^T$ gives us symmetric basis matrices of sublattices $\Lambda_R \subseteq \Lambda \subseteq \Lambda_W$. We also substitute the above basis matrix of the root lattice~$A_R$ by $A_R(Q_R)^{-1}P_R^T$. It is then easy to see that the def\/inition $A_\ell:=P_R D_\ell P_R^T$ gives a centralizer transfer matrix. We calculate it explicitly
	\begin{gather*}
	(A_\ell)_{ij}=\big(P_R D_\ell P_R^{-1}\big)_{ij}=
	\begin{cases}
	\delta_{ij}, &i<n, \\
	\displaystyle (n+1-j)\left( \frac{\ell}{\operatorname{gcd}(\ell,d)} -1 \right) , &i=n \text{ and } j<n, \\
	\displaystyle\frac{\ell}{\operatorname{gcd}(\ell,d)} , & i=j=n. 		
	\end{cases}
	\end{gather*}
Now a form $g$ is uniquely determined by a $d$th root of unity $g(\chi,\chi)=\exp \big(\frac{2 \pi i\cdot k}{d}\big)=\zeta_{d}^k$ with some~$k$. Then we calculate the form $a_g^\ell$ on the generator
	\begin{gather*}
	a_g^\ell(\chi,\chi)=q^{-(\chi,A_\ell(\chi))} g(\chi,A_\ell(\chi))
	= q^{-\frac{(n+1)^2 \cdot \ell}{d^2 \operatorname{gcd}(\ell,\hat{d})}(\lambda_n^\vee,\lambda_n^\vee)} \cdot g(\chi,\chi)^{\frac{ \ell}{\operatorname{gcd}(\ell,\hat{d})}}\\
\hphantom{a_g^\ell(\chi,\chi)}{}
	= \exp \left(\frac{2 \pi i \cdot(k \ell - \hat{d}n)}{d \cdot \operatorname{gcd}(\ell,\hat{d})} \right).
	\end{gather*}	
For example the trivial $g$ (i.e., $k=0$) gives an $R$-matrix for all lattices $\Lambda$ (def\/ined by $\hat{d}d=n+1$) if\/f $\frac{\hat{d}}{\operatorname{gcd}(\ell,\hat{d})}$ is coprime to~$d$. For $\ell$ coprime to the divisor $n+1$ this amounts to all lattices associated to decompositions of $n+1$ into two coprime factors.
\end{Example}

\subsection[Case $D_n$]{Case $\boldsymbol{D_n}$}\label{section4.4}

Finally, we consider the root lattice $D_n$. Since we have $\pi_1(D_{2n\geq 4}) \cong \Z_2 \times \Z_2$ and $\pi_1(D_{2n+1\geq 5}) \cong \Z_4$, it is appropriate to split this investigation in two steps. We start with~$D_{2n\geq 4}$. In order to compute the respective Smith decompositions, we used the software \textit{Wolfram Mathematica}.

\begin{Example} In the case $D_{2n\geq 4}$, we have three dif\/ferent possibilities for the lattices $\Lambda_R \subseteq \Lambda_1,\Lambda_2 \subseteq \Lambda_W$:

1.~$\Lambda_1 \neq \Lambda_2$, $H_1 \cong H_2 \cong \Z_2$: In this case, the subgroups $\Lambda_i/\Lambda_R \subseteq \Lambda_R$ are spanned by the fundamental weights $\lambda_{2(n-1)+i}$. As in the case $A_n$, we def\/ine the centralizer transfer map $A_\ell:=P_R D_\ell P_R^{-1}$ on $H_2$. This is possible since the symmetric basis matrix $A_{\Lambda_2}=P_R S_{\Lambda_2} P_R^T$ of $\Lambda_2$ is already in Smith normal form. Using the software \textit{Wolfram Mathematica} in order to compute $P_R$, we obtain $A_\ell(\lambda_{2n})=\frac{\ell}{\operatorname{gcd}(2,\ell)}$. Combining this with $(\lambda_{2n-1},\lambda_{2n})=\frac{n-1}{2}$, we get
	\begin{gather*}
		a_g^\ell(\lambda_{2n-1},\lambda_{2n})=\exp \left(\frac{2 \pi i\cdot (kl-2(n-1))}{2\cdot \operatorname{gcd}(2,\ell)} \right)
	\end{gather*}
	for $g(\lambda_{2n-1},\lambda_{2n})=\exp \big( \frac{2 \pi i k}{2} \big) $.

 2.~$\Lambda_1 = \Lambda_2$, $H_i \cong \Z_2$: Without restrictions and in order to use the same def\/inition for $A_\ell$ as above, we choose $\Lambda_i$, s.t.\ the group $\Lambda_i/\Lambda_R$ is spanned by $\lambda_{2n}$. Combining the above result $A_\ell(\lambda_{2n})=\frac{\ell}{\operatorname{gcd}(2,\ell)}$ with $(\lambda_{2n},\lambda_{2n})=\frac{n}{2}$, we obtain{\samepage
\begin{gather*}
a_g^\ell(\lambda_{2n},\lambda_{2n})=\exp \left(\frac{2 \pi i\cdot (kl-2n)}{2\cdot \operatorname{gcd}(2,\ell)} \right)
\end{gather*}
for $g(\lambda_{2n},\lambda_{2n})=\exp \big( \frac{2 \pi i k}{2} \big) $.}

3.~$\Lambda_1=\Lambda_2=\Lambda_W$, $H \cong \Z_2 \times \Z_2$: A group pairing $g\colon (\mathbb{Z}_2 \times \mathbb{Z}_2) \times (\mathbb{Z}_2 \times \mathbb{Z}_2) \to \mathbb{C}^\times$ is uniquely def\/ined by a matrix $K \in \mathfrak{gl}(2,\mathbb{F}_2)$, so that
	\begin{gather*}
		g(\lambda_{2(n-1)+i},\lambda_{2(n-1)+j})=\exp \left( \frac{2 \pi i K_{ij}}{2} \right).
	\end{gather*}
	Since $D_n$ is simply-laced, we have $A_\ell=\ell \cdot \operatorname{id}$. Using $(\lambda_{2(n-1)+i},\lambda_{2(n-1)+j}) \text{ mod } 2=\delta_{i+j \text{odd}}$, we obtain
	\begin{gather*}
		a_\ell^g(\lambda_{2(n-1)+i},\lambda_{2(n-1)+j}) =\exp \left( \frac{2 \pi i \cdot K_{ij}\ell}{2}\right)(-1)^{i+j}.
	\end{gather*}
\end{Example}

The last step is the case $D_{2n+1\geq 5}$:
\begin{Example}Since it it is simply-laced and its fundamental group is cyclic, the case $D_{2n+1\geq 5}$ can be treated very similar to $A_n$. We distinguish two cases:

 1.~$\Lambda_1=\Lambda_2$, $H_i =\car{2\lambda_{2n+1}} \cong \Z_2$. As in the case $A_n$, we def\/ine the centralizer transfer map $A_\ell:=P_R D_\ell P_R^{-1}$ on $H_2$. Using $(\lambda_{2n+1},\lambda_{2n+1})=\frac{2n+1}{4}$, we obtain
		\begin{gather*}
			a_g^\ell(2\lambda_{2n+1},2\lambda_{2n+1})=\exp \left(\frac{2 \pi i\cdot (k\ell-2(2n+1))}{2\cdot \operatorname{gcd}(2,\ell)} \right).
		\end{gather*}
		for $g(2\lambda_{2n+1},2\lambda_{2n+1})=\exp \big( \frac{2 \pi i k}{2} \big) $.

2.~$\Lambda_1=\Lambda_2=\Lambda_W$, $H =\car{\lambda_{2n+1}} \cong \Z_4$. By an analogous argument as above, we obtain
		\begin{gather*}
			a_g^\ell(\lambda_{2n+1},\lambda_{2n+1})=\exp \left(\frac{2 \pi i\cdot (k\ell-(2n+1))}{4} \right)
		\end{gather*}
		for $g(\lambda_{2n+1},\lambda_{2n+1})=\exp \big( \frac{2 \pi i k}{4} \big) $.
\end{Example}

\subsection{Table of all quasitriangular quantum groups}\label{section4.5}
In the following table, we list all simple Lie algebras and check for which non-trivial choices of $\Lambda$, $\Lambda_i$, $\ell$ and $g$ the element $R_0\bar{\Theta}$ is an $R$-matrix. As before, we def\/ine $H_i:=\Lambda_i/\Lambda_R$ and $H:=\Lambda/\Lambda_R$. In the cyclic case, if $x_i$ are generators of the~$H_i$, then the pairing is uniquely def\/ined by an element $1 \leq k \leq |H_i|$, s.t.\ $g(x_1,x_2)=\exp \big( \frac{2 \pi i k}{|H_i|} \big)$. In the case~$D_{2n}$, $\Lambda=\Lambda_W$, $g$ is uniquely def\/ined by a $2\times 2$-matrix $K \in \mathfrak{gl}(2,\F_2)$, s.t.\ $g(\lambda_{2(n-1)+i},\lambda_{2(n-1)+j})=\exp \big( \frac{2 \pi i K^g_{ij}}{2} \big)$ for $i,j\in \{1,2\}$.

 The columns of the following table are labeled by
 \begin{enumerate}\itemsep=0pt
 	\item[1)] the f\/inite-dimensional simple complex Lie algebra~$\g$,
 	\item[2)] the natural number~$\ell$, determining the root of unity $q=\exp\big( {\frac{2 \pi i}{\ell}}\big) $,
 	\item[3)] the number of possible $R$-matrices for the Lusztig ansatz,
 	\item[4)] the subgroups $H_i \subseteq H=\Lambda/\Lambda_R$ introduced in Theorem \ref{thm:solutionsgrpeq},
 	\item[5)] the subgroups $H_i$ in terms of generators given by multiples of fundamental dominant weights $\lambda_i \in \Lambda_W$,
 	\item[6)] the group pairing $g\colon H_1 \times H_2 \to \C^\times$ determined by its values on generators,
 	\item[7)] the group pairing $a_g^\ell \subseteq \Lambda/\Lambda'$ introduced in Theorem~\ref{thm:radical} determined by its values on generators.
 \end{enumerate}

{\footnotesize 
\centering
\begin{longtable}{@{}c@{\,}||@{\,}c@{\,}|@{\,}c@{\,}||@{\,}c@{\,}|@{\,}c@{\,}|@{\,}c@{\,}|@{\,}c@{}}
	$\g$
	&$\ell$
	&\#
	&$H_i\cong$
	&$H_i\,{\scriptstyle (i=1,2)}$
	&$g$
	&$a_g^\ell$
	\\ \hline \hline
	\multirow{2}{*}{\text{all}} 	& 	& \multirow{2}{*}{1}	& \multirow{2}{*}{$\Z_1$} &\multirow{2}{*}{$\car{0}$} &\multirow{2}{*}{$g=1 $} &\multirow{2}{*}{$1$} 	\\
	&&&&&&\\	
	\hline
	\multirow{4}{*}{} 	&	& \multirow{4}{*}{$\infty$}	& \multirow{2}{*}{$\Z_d$} &\multirow{2}{*}{$\car{\hat{d}\lambda_n}$} &\multirow{2}{*}{$g(\hat{d}\lambda_n,\hat{d}\lambda_n)=\exp \big( \frac{2\pi i k}{d}\big) $} &\multirow{4}{*}{$\exp \big(\frac{2 \pi i \cdot(k \ell - \hat{d}n)}{d \cdot \operatorname{gcd}(\ell,\hat{d})} \big)$} 	\\
	$A_{n\geq 1}$		&	& 	&&&&\\
	$\pi_1=\Z_{n+1}$	& 	&	& \multirow{2}{*}{$d\,|\, n+1$}&\multirow{2}{*}{$\hat{d}=\frac{n+1}{d}$}&\multirow{2}{*}{$\operatorname{gcd}\big( d,\frac{k \ell - \hat{d}n}{\operatorname{gcd}(\ell,\hat{d})}\big) =1 $}&\\
	&&&&&&\\	
	\hline
	\multirow{4}{*}{} 	& \multirow{2}{*}{$\ell$ even}	& \multirow{2}{*}{2}	& \multirow{4}{*}{$\Z_2$} &\multirow{4}{*}{$\car{\lambda_n}$} &\multirow{2}{*}{$g(\lambda_n,\lambda_n)=\pm 1 $} &\multirow{2}{*}{$-1$} 	\\
	$B_{n\geq 2}$		&	& 	&&&&\\ \cline{2-3} \cline{6-7}
	$\pi_1=\Z_{2}$	& \multirow{2}{*}{$\ell$ odd} 	& \multirow{2}{*}{1}	&&&\multirow{2}{*}{$g(\lambda_n,\lambda_n)=(-1)^{n+1}$}& \multirow{2}{*}{$\exp \big(\frac{2 \pi i \cdot(k \ell - n)}{2} \big)$}\\
	&&&&&&\\	
	\hline
	\multirow{6}{*}{} 	& \multirow{2}{*}{$\ell\equiv2\text{ mod }4$}	& \multirow{2}{*}{1}	& \multirow{6}{*}{$\Z_2$} &\multirow{6}{*}{$\car{\lambda_n}$} &\multirow{2}{*}{$g(\lambda_n,\lambda_n)=1 $} &\multirow{2}{*}{$\exp \big(\frac{2 \pi i \cdot(k \frac{\ell}{2} + 1)}{2} \big)$} 	\\
	& 	& 	&&&&\\	\cline{2-3} \cline{6-7}
	$C_{n\geq 3}$		& \multirow{2}{*}{$\ell\equiv0\text{ mod }4$}	& \multirow{2}{*}{2}	&&&\multirow{2}{*}{$g(\lambda_n,\lambda_n)=\pm 1 $} &\multirow{2}{*}{$-1$} \\
	$\pi_1=\Z_{2}$		& 	& 	&&&&	\\ \cline{2-3} \cline{6-7}
	& \multirow{2}{*}{$\ell$ odd}	& \multirow{2}{*}{1}	&&&\multirow{2}{*}{$g(\lambda_n,\lambda_n)=- 1 $}&\multirow{2}{*}{$\exp \big(\frac{2 \pi i \cdot(k \ell - 2n)}{2} \big)$}\\
	&&&&&&\\	
	\hline
	\multirow{18}{*}{} 	& \multirow{2}{*}{$\ell \equiv 2 \text{ mod }4$}	&\multirow{2}{*}{1} 	& \multirow{6}{*}{$\Z_2$} &\multirow{3}{*}{$H_1 \cong\car{\lambda_{2n-1}}$} & \multirow{2}{*}{$g(\lambda_{2n-1},\lambda_{2n})=(-1)^n$} &\multirow{4}{*}{$\exp \big(\frac{2 \pi i\cdot (k\frac{\ell}{2}-n+1))}{2} \big)$}	\\
	&&&&&&\\	\cline{2-3} \cline{6-6}
	&\multirow{2}{*}{$\ell \equiv 0 \text{ mod }4$}	& 	\multirow{2}{*}{$2\delta_{2 \,|\, n}$}& &&\multirow{2}{*}{$g(\lambda_{2n-1},\lambda_{2n})=\pm 1$, $n$ even} & \\
	& 	&	& &\multirow{3}{*}{$H_2 \cong\car{\lambda_{2n}}$}&&	\\ \cline{2-3} \cline{6-7}
	&\multirow{2}{*}{$\ell$ odd} &\multirow{2}{*}{1}&&&\multirow{2}{*}{$g(\lambda_{2n-1},\lambda_{2n})=-1$}&\multirow{2}{*}{$\exp \big(\frac{2 \pi i\cdot (k\ell-2(n-1))}{2} \big)$}\\
	$D_{2n\geq 4}$&&&&&&\\\cline{2-7}
	$\pi_1=\Z_{2} \times \Z_2$&\multirow{2}{*}{$\ell \equiv 2 \text{ mod }4$}&\multirow{2}{*}{$1$}&\multirow{6}{*}{$\Z_2$}&\multirow{6}{*}{$\car{\lambda_{2n}}$}&\multirow{2}{*}{$g(\lambda_{2n},\lambda_{2n})=(-1)^{n+1}$}&\multirow{4}{*}{$\exp \big( \frac{2 \pi i (k \frac{\ell}{2}-n)}{2} \big) $}\\	
	&&&&&&\\\cline{2-3} \cline{6-6}
	&\multirow{2}{*}{$\ell \equiv 0 \text{ mod }4$}&\multirow{2}{*}{$2\delta_{2 \nmid n}$}&&&\multirow{2}{*}{$g(\lambda_{2n},\lambda_{2n})=\pm 1$, $n$ odd}&\\
	&&&&&&\\\cline{2-3} \cline{6-7}
	&\multirow{2}{*}{$\ell$ odd}&\multirow{2}{*}{1}&&&\multirow{2}{*}{$g(\lambda_{2n},\lambda_{2n})=-1$}&\multirow{2}{*}{$\exp \big( \frac{2 \pi i (k \ell-2n)}{2} \big) $}\\
	&&&&&&\\\cline{2-7}
	&\multirow{2}{*}{$\ell$ even}	&\multirow{2}{*}{$16$}&\multirow{4}{*}{$\Z_2 \times \Z_2$}&\multirow{4}{*}{$\car{\lambda_{2n},\lambda_{2n+1}}$}&\multirow{2}{*}{$g(\lambda_{2(n-1)+i},\lambda_{2(n-1)+j})=\pm 1$}&\multirow{4}{*}{$\exp \big( \frac{2 \pi i \cdot K_{ij}\ell}{2}\big)(-1)^{i+j}$}\\
	&&&&&&\\\cline{2-3}\cline{6-6}
	&\multirow{2}{*}{$\ell$ odd}&\multirow{2}{*}{$6$}&&&\multirow{2}{*}{$\det(K)=K_{12}+K_{12}$ mod $2$}&\\
	&&&&&&\\
	\hline
	\multirow{10}{*}{} 	& \multirow{2}{*}{$\ell\equiv2\text{ mod }4$}	& \multirow{2}{*}{ 1}	& \multirow{6}{*}{$\Z_2$} &\multirow{6}{*}{$\car{2\lambda_{2n+1}}$} &\multirow{2}{*}{$g(2\lambda_{2n+1},2\lambda_{2n+1})=1 $} &\multirow{4}{*}{$\exp \big(\frac{2 \pi i \cdot(k \frac{\ell}{2} -2n-1)}{2} \big)$} 	\\
	& 	& 	&&&&\\	\cline{2-3} \cline{6-6}
	& \multirow{2}{*}{$\ell\equiv0\text{ mod }4$}	& \multirow{2}{*}{2}	&&&\multirow{2}{*}{$g(2\lambda_{2n+1},2\lambda_{2n+1})=\pm 1 $} & \\
	& 	& 	&&&&	\\ \cline{2-3} \cline{6-7}
	$D_{2n+1\geq 5}$	& \multirow{2}{*}{$\ell$ odd}	& \multirow{2}{*}{1}	&&&\multirow{2}{*}{$g(2\lambda_{2n+1},2\lambda_{2n+1})=- 1 $}&\multirow{2}{*}{$\exp \big(\frac{2 \pi i \cdot(k \ell - 2(2n+1))}{2} \big)$}\\
	$\pi_1=\Z_{4}$	&&&&&&\\ \cline{2-7}
	&\multirow{2}{*}{$\ell$ even}&\multirow{2}{*}{4}&\multirow{4}{*}{$\Z_4$}&\multirow{4}{*}{$\car{\lambda_{2n+1}}$}&\multirow{2}{*}{$g(\lambda_{2n+1},\lambda_{2n+1})=c$, $c^4=1$}&\multirow{4}{*}{$\exp \big(\frac{2 \pi i \cdot(k \ell - (2n+1))}{4} \big)$}\\
	&&&&&&\\ \cline{2-3}\cline{6-6}
	&\multirow{2}{*}{$\ell$ odd}&\multirow{2}{*}{2}&&&\multirow{2}{*}{$g(\lambda_{2n+1},\lambda_{2n+1})=\pm 1$}&\\
	&&&&&&\\
	\hline
	\multirow{6}{*}{} 	& \multirow{2}{*}{$\ell \equiv 0$ mod $3$}	& \multirow{2}{*}{3 }	& \multirow{6}{*}{$\Z_3$} &\multirow{6}{*}{$\car{\lambda_n}$} &\multirow{2}{*}{$g(\lambda_n,\lambda_n)=c$, $c^3=1$} &\multirow{6}{*}{$\exp \big(\frac{2 \pi i \cdot(k \ell -1)}{3} \big)$} 	\\
	&	& 	&&&&\\ \cline{2-3} \cline{6-6}
	$E_{6}$	& \multirow{2}{*}{$\ell \equiv 1$ mod $3$} 	& \multirow{2}{*}{2 }	&&&\multirow{2}{*}{$g(\lambda_n,\lambda_n)=1,\exp \big(\frac{2 \pi i 2}{3} \big) $}&\\
	$\pi_1=\Z_{3}$&&&&&&\\ \cline{2-3} \cline{6-6}
	&\multirow{2}{*}{$\ell \equiv 2$ mod $3$}&\multirow{2}{*}{2 }&&&\multirow{2}{*}{$g(\lambda_n,\lambda_n)=1,\exp \big(\frac{2 \pi i }{3} \big) $}&\\	
	&&&&&&\\	
	\hline
	\multirow{4}{*}{} 	& \multirow{2}{*}{$\ell$ even}	& \multirow{2}{*}{2 }	& \multirow{4}{*}{$\Z_2$} &\multirow{4}{*}{$\car{\lambda_n}$} &\multirow{2}{*}{$g(\lambda_n,\lambda_n)=\pm 1 $} &\multirow{4}{*}{$\exp \big(\frac{2 \pi i \cdot(k \ell -1)}{2} \big)$}\\
	$E_{7}$		&	& 	&&&&\\ \cline{2-3} \cline{6-6}
	$\pi_1=\Z_{2}$	& \multirow{2}{*}{$\ell$ odd} 	& \multirow{2}{*}{1 }	&&&\multirow{2}{*}{$g(\lambda_n,\lambda_n)=1$}& \\
	&&&&&&\\	
	\hline		
	\caption{Solutions for $R_0$-matrices.}	\label{tbl:Solutions}
\end{longtable}

}

The Lie algebras $E_8$, $F_4$ and $G_2$ have trivial fundamental groups and thus have no non-trivial solution. We want to emphasize once more that the choice $\Lambda_i=\Lambda_R$ always leads to a~quasitriangular quantum group.

The following lemma connects our results with Lusztig's original result:
\begin{Lemma}\label{lm:Lusztigkernel}
In Lusztig's definition of a quantum group he uses the quotient
\begin{gather*}
\Lambda'_{\rm Lusz}=2\operatorname{Cent}_{\Lambda_R}(2\Lambda_W).
 \end{gather*} This coincide with our choice $\Lambda'=\operatorname{Cent}_{\Lambda_R}(\Lambda_1 + \Lambda_2)$, if and only if
	\begin{gather}\label{LusztigsChiceLambda'equation}
	2\gcd\big(\ell,d_i^\Lambda\big)=\gcd\big(\ell,2d_i^W\big),
	\end{gather}
	where the $d_i^\Lambda$ denote the invariant factors of $\Lambda_W^\vee/\Lambda$ and the $d_i^W$ denote the invariant factors of $\Lambda_W^\vee/\Lambda_W$ $($i.e., ordered root lengths$)$.
	
In particular, for $\ell$ odd these choices never coincide. For $\Lambda=\Lambda_W$, $\Lambda'=\Lambda'_{\rm Lusz}$ holds if and only if $2d_i \,|\, \ell$. This is the most extreme case of divisibility and it is precisely the case appearing in logarithmic conformal field theories.
\end{Lemma}
\begin{proof}We f\/irst note that in our cases, $\Lambda'=\operatorname{Cent}_{\Lambda_R}(\Lambda_1 + \Lambda_2)=\operatorname{Cent}_{\Lambda_R}(\Lambda)$. We have
\begin{gather*}
2\operatorname{Cent}_{\Lambda_R}(2\Lambda_W) = 2\big(\Lambda_R \cap \widehat{2\Lambda_W}\big) \\
\hphantom{2\operatorname{Cent}_{\Lambda_R}(2\Lambda_W)}{} =A_R2\left(\Lambda_W^\vee \cap A_W^{-1} \frac{\ell}{2}\Lambda_W^\vee\right) =A_R \operatorname{Diag}\left( \frac{2 \ell}{\gcd(\ell,2d_i^W)}\right) \Lambda_W^\vee. 		
	\end{gather*}
By Lemma \ref{lm:ExplicitFormOfCentralizers}, this coincides with $\Lambda'$ if and only if equation~(\ref{LusztigsChiceLambda'equation}) holds.
\end{proof}

\section[Factorizability of quantum group $R$-matrices]{Factorizability of quantum group $\boldsymbol{R}$-matrices}\label{section5}

We f\/irst recall the def\/inition of factorizable braided tensor categories and factorizable Hopf algebras, respectively.

\begin{Definition}[\cite{EGNO15}]\label{def:FactorizabilityOfCats}
A braided tensor category $\mathcal{C}$ is \emph{factorizable} if the canonical braided tensor functor $G\colon \mathcal{C}\boxtimes \mathcal{C}^\text{op} \to \mathcal{Z}(\mathcal{C})$ is an equivalence of categories.
\end{Definition}

In \cite{Sch01}, Schneider gave a dif\/ferent characterization of factorizable Hopf algebras in terms of its Drinfeld double, leading to the following theorem:

\begin{Definition}\label{def:FactorizableHopfAlgs}A f\/inite-dimensional quasitriangular Hopf algebra $(H,R)$ is called \emph{factorizable} if its \emph{monodromy matrix} $M:=R_{21}\cdot R \in H\otimes H$ is non-degenerate, i.e., the following linear map is bijective
	\begin{gather*}
		H^* \to H,\qquad \phi\mapsto (\operatorname{id}\otimes \phi)(M).
	\end{gather*}
	Equivalently, this means we can write $M=\sum_i R_1^i\otimes R_2^i$ for two basis' $R_1^i,R_2^j \in H$.
\end{Definition}
\begin{Theorem}\label{thm:FactorizabilityOfCats}
Let $(H,R)$ be a finite-dimensional quasitriangular Hopf algebra. Then the ca\-te\-gory of finite-dimensional $H$-modules $H-\mathsf{mod}_{fd}$ is factorizable if and only if $(H,R)$ is a~factorizable Hopf algebra.
\end{Theorem}	
	
Shimizu~\cite{Shi16} has recently proven a number of equivalent characterizations of factorizability for arbitrary (in particular non-semisimple) braided tensor categories. Besides the two previous characterizations (equivalence to Drinfeld center and nondegeneracy of the monodromy matrix), factorizability is equivalent to the fact that the so-called transparent objects are all trivial, see Theorem~\ref{thm_Shimizu} below, which will become visible during our analysis later.
	
\subsection[Monodromy matrix in terms of $R_0$]{Monodromy matrix in terms of $\boldsymbol{R_0}$}\label{section5.1}

In order to obtain conditions for the factorizability of the quasitriangular small quantum groups $(u_q(\g,\Lambda,\Lambda'),R_0(f)\bar{\Theta})$ as in Theorem \ref{thm:R0} in terms of $\g$, $q$, $\Lambda$ and $f$, we start by calculating the monodromy matrix $M:=R_{21} \cdot R \in
u_q(\g,\Lambda,\Lambda')\otimes u_q(\g,\Lambda,\Lambda')$ in general as far as possible:
\begin{Lemma}
 For $R=R_0(f)\bar\Theta$ as in Theorem~{\rm \ref{thm:R0}}, the factorizability of $R$ is equivalent to the invertibility of the following
 complex-valued matrix~$m$ with entries indexed by elements in $\mu,\nu\in \Lambda/\Lambda'$:
 \begin{gather*}
 	m_{\mu,\nu}:=\sum_{\mu',\nu'\in\Lambda/\Lambda'}f(\mu-\mu',\nu-\nu')f(\nu',\mu').
 \end{gather*}
\end{Lemma}
\begin{proof}
 We f\/irst plug in the expressions for $R_0$ from Theorem~\ref{thm:solutionsgrpeq} and $\bar\Theta$ from Theorem~\ref{thm:R0} and
 simplify:
 \begin{gather*}
 M:=R_{21} \cdot R =(R_0)_{21}\cdot \bar\Theta_{21} \cdot R_0\cdot \bar\Theta\\
 \hphantom{M}{} =\left(\sum_{\mu_1,\nu_1\in \Lambda} f(\mu_1,\nu_1)K_{\nu_1}\otimes
 K_{\mu_1}\right) \left( \sum_{\beta_1\in\Lambda_R^+}(-1)^{{\rm tr} \beta_1} q_{\beta_1}
 \sum_{b_1\in B_{\beta_2}}b_1^{*+}\otimes b_1^- \right)\\
 \hphantom{M=}{}
 \times \left(\sum_{\mu_2,\nu_2\in \Lambda} f(\mu_2,\nu_2)K_{\mu_2}\otimes
 K_{\nu_2}\right) \left( \sum_{\beta_2\in\Lambda_R^+}(-1)^{{\rm tr} \beta_2} q_{\beta_2}
 \sum_{b_2\in B_{\beta_2}}b_2^-\otimes b_2^{*+}\right)\\
 \hphantom{M}{}
 =\sum_{\beta_1,\beta_2\in\Lambda_R^+}\!(-1)^{{\rm tr} \beta_1+\beta_2} q_{\beta_1}q_{\beta_2}\! \left(\sum_{\mu_1,\mu_2,\nu_1,\nu_2\in \Lambda}\!
 f(\mu_1,\nu_1)f(\mu_2,\nu_2)q^{\beta_1(\nu_2-\mu_2)}K_{\nu_1+\mu_2}\otimes K_{\mu_1+\nu_2} \right)\! \\
 \hphantom{M=}{}
 \times \left(\sum_{b_1\in B_{\beta_1},b_2\in B_{\beta_2}} b_1^{*+}b_2^{-}\otimes b_1^{-}b_2^{*+}\right),
 \end{gather*}
 where $\Lambda_R^+=\N_0[\Delta]$. The last equation holds since $b_1^-\in u_{\beta_1}^-$ and hence fulf\/ills $K_{\nu_2}b_1^-=q^{-\beta_1\nu_2}b_1^-K_{\nu_2}$ and similarly for~$b_1^{*+}$. We have two triangular decompositions
 \begin{gather*}
 	 u_q=u_q^0u_q^-u_q^+,\qquad u_q=u_q^0u_q^+u_q^- ,
 \end{gather*}
 and the $\Lambda_R^+$-gradation on $u_q^{\pm}$ induces a gradation
 \begin{gather*}
 	u_q \otimes u_q \cong \bigoplus_{\beta_1,\beta_2} \big(u^0\otimes u^0\big)\big({u_q^+}_{\beta_1}{u_q^-}_{\beta_2}\otimes {u_q^-}_{\beta_1}{u_q^+}_{\beta_2}\big).
 \end{gather*}
 The factorizability of $R$ is equivalent to the invertibility of $M$ interpreted as a matrix indexed by the PBW basis. The grading implies a block matrix form of $M$, so the invertibility $M$ is equivalent to the invertibility of $M^{\beta_1,\beta_2} \in (u_q \otimes u_q)_{(\beta_1,\beta_2)}$ for every $\beta_1,\beta_2 \in \Lambda_R^+$ as follows
 \begin{gather*}
 	M^{\beta_1,\beta_2}:= \left(\sum_{\mu_1,\mu_2,\nu_1,\nu_2\in \Lambda}
 	f(\mu_1,\nu_1)f(\mu_2,\nu_2)q^{\beta_1(\nu_2-\mu_2)}K_{\nu_1+\mu_2}\otimes K_{\mu_1+\nu_2} \right)\\
 \hphantom{M^{\beta_1,\beta_2}:=}{}\times
 	\left(\sum_{b_1\in B_{\beta_1},b_2\in B_{\beta_2}} 	b_1^{*+}b_2^{-}\otimes b_1^{-}b_2^{*+}\right).
 \end{gather*}
 Since the second sum in $M^{\beta_1,\beta_2}$ runs over a basis in ${u_q^+}_{\beta_1}{u_q^-}_{\beta_2}\otimes {u_q^-}_{\beta_1}{u_q^+}_{\beta_2}$,
 the invertibility of~$M$ is equivalent to the invertibility for all $\beta_1 \in\Lambda_R^+$ the following element:
 \begin{gather*}
 M_0^{\beta_1}:=\sum_{\mu_1,\mu_2,\nu_1,\nu_2\in \Lambda/\Lambda'}q^{\beta_1(\nu_2-\mu_2)}
 f(\mu_1,\nu_1)f(\mu_2,\nu_2)K_{\nu_1+\mu_2}\otimes K_{\mu_1+\nu_2} \\
 \hphantom{M_0^{\beta_1}}{} =\sum_{\mu,\nu\in\Lambda/\Lambda'} K_\nu \otimes K_\mu
 \cdot \left(\sum_{\mu',\nu'\in\Lambda/\Lambda'} q^{\beta_1(\mu'-\nu')}
 f(\mu-\mu',\nu-\nu')f(\nu',\mu')\right).
 \end{gather*}
 Since $K_\nu\otimes K_\mu$ is a vector space basis of $u_q^0\otimes u_q^0=\C[\Lambda/\Lambda']\otimes\C[\Lambda/\Lambda']$, this in turn is
 equivalent to the invertibility of the following family of matrices $m^{\beta_1}$ for all $\beta_1\in\Lambda_R^+$ with rows/columns indexed by elements in $\mu,\nu\in \Lambda/\Lambda'$:
 \begin{gather*}
 	m^{\beta_1}_{\mu,\nu}:=\sum_{\mu',\nu'\in\Lambda/\Lambda'}	f(\mu-\mu',\nu-\nu')f(\nu',\mu')q^{\beta_1(\mu'-\nu')}.
 \end{gather*}
 We now use the fact that $R$ was indeed an $R$-matrix. By property~(\ref{f01}) in Theorem~\ref{thm:R0} we have
 \begin{gather*}
 	m^{\beta_1}_{\mu,\nu}=\sum_{\mu',\nu'\in\Lambda/\Lambda'}f(\mu-\mu',\nu-\nu')f(\nu'+\beta_1,\mu')q^{-\beta_1\nu'}.
 \end{gather*}
Since the invertibility of a matrix $m_{\mu,\nu}$ is equivalent to the invertibility of any matrix $m_{\mu,\nu+\beta_1}$, we may substitute $\nu'\mapsto \nu'+\beta_1$, $\nu\mapsto \nu+\beta_1$, pull the constant factor $q^{-\beta_1^2}$ in front (which also does not af\/fect invertibility) and hence eliminate the f\/irst $\beta_1$ from the condition. Hence the invertibility of $R$ is equivalent to the invertibility of the following family of matrices $\tilde{m}^{\beta_1}$ for all $\beta_1\in\Lambda_R^+$:
 \begin{gather*}
 	\tilde{m}^{\beta_1}_{\mu,\nu}:=\sum_{\mu',\nu'\in\Lambda/\Lambda'}	f(\mu-\mu',\nu-\nu')f(\nu',\mu')q^{-\beta_1\nu'}.
 \end{gather*}
We may now use the same procedure to eliminate the second $\beta_1$, hence the invertibility of $R$ is equivalent to the invertibility of the following matrix with rows/columns induced by elements in $\mu,\nu\in \Lambda/\Lambda'$:	
 \begin{gather*}
 	m_{\mu,\nu}:=\sum_{\mu',\nu'\in\Lambda/\Lambda'}f(\mu-\mu',\nu-\nu')f(\nu',\mu').
 \end{gather*}
 This was the assertion we wanted to prove.
\end{proof}

\begin{Definition}Let $g\colon G_1 \times G_2 \to \mathbb{C}^\times$ be a group pairing. It induces a symmetric form on the product $G_1 \times G_2$ we denote by $\operatorname{Sym}(g)$:
	\begin{align*}
	\operatorname{Sym}(g)\colon \ & (G_1 \times G_2)^{\times 2} \longrightarrow \mathbb{C}^\times, \\
	& ((\mu_1,\mu_2),(\nu_1,\nu_2)) \longmapsto g(\mu_1,\nu_2)g(\nu_1,\mu_2).
	\end{align*}
\end{Definition}

\begin{Lemma}\label{lm:Sym(f)nondeg}
	If $g\colon G_1 \times G_2 \to \mathbb{C}^\times$ is a perfect pairing of abelian groups, then the symmetric form $\operatorname{Sym}(g)$ is perfect.
\end{Lemma}
\begin{proof}By assumption, $g \times g$ def\/ines an isomorphism between $G_1 \times G_2$ to $\widehat{G_2} \times \widehat{G_1}$. The symmetric form $\operatorname{Sym}(g)$ is given by the composition of this isomorphism with the canonical isomorphism $\widehat{G_2} \times \widehat{G_1} \cong \widehat{G_1 \times G_2}$. This proves the claim.
\end{proof}

Consider for a f\/inite abelian group $G$ and subgroups $G_1,G_2 \leq G$ the canonical exact sequence
\begin{gather}\label{al:exactSequence}
0\to G_1 \cap G_2 \to G_1 \times G_2\to G_1+G_2 \to 0.
\end{gather}
For $\mu \in G_1+G_2$, we denote its f\/iber by
\begin{gather*}
	(G_1 \times G_2)_{\mu}:=\{(\mu_1,\mu_2)\in G_1 \times G_2 \,|\, \mu_1+\mu_2=\mu \}.
\end{gather*}
Moreover, we def\/ine
\begin{gather*}
\operatorname{Rad}:= \big\{ (\mu_1,\mu_2) \in G_1 \times G_2 \,|\, \operatorname{Sym}\big(\hat{f}\big)((\mu_1,\mu_2),x)=1 \; \forall\, x \in (G_1 \times G_2)_0 \big\}, \\
\operatorname{Rad}_\mu:= \operatorname{Rad} \cap (G_1 \times G_2)_{\mu}, \\
\operatorname{Rad}_0^\perp := \{ \mu_1 + \mu_2 \in G \,|\, (\mu_1,\mu_2) \in \operatorname{Rad} \}.
\end{gather*}

\begin{Lemma}We have two split exact sequences:
		\begin{gather*}
		0\to \operatorname{Rad}_0 \to \operatorname{Rad}\to \operatorname{Rad}_0^\perp \to 0,\\
		0\to \operatorname{Rad}_0^\perp \to G \to \operatorname{Rad}_0 \to 0.
		\end{gather*}
\end{Lemma}
\begin{proof}The f\/irst sequence is exact by def\/inition of the three groups. Moreover, we know
\begin{gather*}
\operatorname{Rad} = \ker\big(\hat{\iota} \circ \operatorname{Sym}\big(\hat{f}\big)\big) \cong \ker(\hat{\iota}) = \text{im}(\hat{\pi}) \cong \hat{G} \cong G,
\end{gather*}
where $\hat{\iota}$, $\hat{\pi}$ denote the duals of the inclusion and projection in~(\ref{al:exactSequence}). In Example~\ref{ex:RadSymf} we will see that in the case $G_1=G_2=G$, $\hat{f}$ symmetric, $\operatorname{Rad}_0$ is the $2$-torsion subgroup of~$G$, and the second map in the second exact sequence is just the projection, hence both diagrams split in this case. If $\hat{f}$ is asymmetric, we will see in Section~\ref{section5.3} that $\operatorname{Rad}_0$ is isomorphic to $\Z_2^k$ for some $k\geq 2$, thus
\begin{gather*}
\operatorname{Rad}_0^\perp \longrightarrow \operatorname{Rad}, \qquad x \longmapsto \sum_{\tilde{x} \in \operatorname{Rad}_x} \tilde{x}
\end{gather*}
is a section of the f\/irst exact sequence. Here we used that the sum over all elements in $\Z_2^k$ vanishes. Again, it follows that both diagrams split. Finally, if $G_1 \neq G_2$ (i.e., in the case $D_{2n}$), then $\hat{f}=q^{-(\cdot,\cdot)}$ on $G_1 \cap G_2$. By the same argument as in Example~\ref{ex:RadSymf}, $\operatorname{Rad}_0$ is the $2$-torsion subgroup of $G_1 \cap G_2$. But we have $G \cong G_1 \cap G_2 \times \pi_1 $ in this case, hence both sequences split.
\end{proof}

\begin{Corollary}Using the projection $\alpha\colon G \to \operatorname{Rad}_0^\perp$ and the inclusion $\beta\colon \operatorname{Rad}_0^\perp \to \operatorname{Rad}$ from the above lemma, we can define a symmetric form on~$G$:
	\begin{align*}
	\operatorname{Sym}_G\big(\hat{f}\big)\colon \ & G \times G \longrightarrow \C^\times, \\
	& (\mu,\nu) \longmapsto \operatorname{Sym}\big(\hat{f}\big)(\beta\circ\alpha(\mu),\beta\circ\alpha(\nu)).
	\end{align*}
	Moreover, we have $\operatorname{Rad}\big(\operatorname{Sym}_G\big(\hat{f}\big)\big) \cong \operatorname{Rad}_0$.
\end{Corollary}

\begin{Theorem}\label{thm:InvertyMonodromymat}
We have shown in Theorem~{\rm \ref{thm:R0}} and Lemma~{\rm \ref{lm:NondegGroupPairing}} that the assumption that $R=R_0(f)\bar\Theta$ is an $R$-matrix is equivalent to the existence of subgroups $G_1,G_2\subset \Lambda/\Lambda'$ of same order some~$d|\Lambda_R/\Lambda'|$ and~$f$ restricting up to a scalar to a non-degenerate pairing $\hat{f}\colon G_1\times G_2\to\C^\times$ and $f$ vanishes otherwise.

	In this notation the matrix $m$ as defined in the previous lemma can be rewritten as
\begin{gather*}
m_{\mu,\nu}=\frac{1}{d^2|\Lambda_R/\Lambda'|^2} \sum_{\substack{\tilde{\mu} \in (G_1 \times G_2)_{\mu}\\\tilde{\nu} \in (G_1 \times G_2)_{\nu}}} \operatorname{Sym}\big(\hat{f}\big)(\tilde{\mu},\tilde{\nu}).
\end{gather*}
It is invertible if and only if $\operatorname{Rad}_0=0$. In this case,
\begin{gather*}
m_{\mu,\nu} = \frac{|G_1 \cap G_2|}{d^2|\Lambda_R/\Lambda'|^2} \operatorname{Sym}_G\big(\hat{f}\big).
\end{gather*}
\end{Theorem}
We f\/irst note that $\operatorname{Rad}_0=0$ implies $\operatorname{Rad}_0^\perp = G$ and thus $G = G_1 + G_2$. Together with Corollary~\ref{cor:LambdaPrime} this implies
\begin{Corollary}
	\begin{gather*}
		\Lambda'=\operatorname{Cent}_{\Lambda_R}(\Lambda).
	\end{gather*}
\end{Corollary}
Before we proof the theorem, we f\/irst give a simple example:
\begin{Example} \label{ex:RadSymf}
Let $G_1=G_2=G$ (correspondingly $\Lambda_1=\Lambda_2=\Lambda$) and assume $\hat{f}$ is symmetric non-degenerate, then the radical measures $2$-torsion:
\begin{gather*}
\operatorname{Rad}\big(\operatorname{Sym}_G\big(\hat{f}\big)\big) \cong \operatorname{Rad}_0=\{\mu\in G\,|\, 2\mu=0\}.
\end{gather*}
Again, this is the only case appearing for cyclic fundamental groups. Hence in all cases except $\g=D_{2n}$ factorizability is equivalent to $|\Lambda/\Lambda'|$ being odd.
	\end{Example}

\begin{proof}[Proof of Theorem~\ref{thm:InvertyMonodromymat}] The f\/irst part of the theorem follows by applying Lemma~\ref{lm:NondegGroupPairing} to the matrix $m$ as given in the previous lemma. Now, assume that $m$ is invertible. We must have $G=G_1+G_2$, otherwise the matrix has zero-columns and rows, dif\/ferently formulated: the f\/ibers $(G_1 \times G_2)_{\mu}$ in the short exact sequence must be non-empty for all $\mu \in G$. If on the other hand, $\operatorname{Rad}_0=0$, then $\operatorname{Rad}_0^\perp = G$ and thus $G_1+G_2=G$ must also hold, thus we assume this from now on. By the short exact sequence the f\/iber $(G_1 \times G_2)_{0} \cong G_1 \cap G_2$, other f\/ibers are of the explicit form $\tilde{\mu}+(G_1 \times G_2)_{0}$ for some choice of representative $\tilde{\mu}$. Therefore,
\begin{gather*}
m_{\mu,\nu}= \frac{1}{d^2|\Lambda_R/\Lambda'|^2} \sum_{\substack{\tilde{\mu} \in (G_1 \times G_2)_{\mu}\\\tilde{\nu} \in (G_1 \times G_2)_{\nu}}} \operatorname{Sym}\big(\hat{f}\big)(\tilde{\mu},\tilde{\nu})\\
\hphantom{m_{\mu,\nu}}{}
= \frac{1}{d^2|\Lambda_R/\Lambda'|^2} \sum_{\tilde{\nu} \in (G_1 \times G_2)_{\nu}} \operatorname{Sym}\big(\hat{f}\big)(\tilde{\mu},\tilde{\nu}) \sum_{\tilde{\eta} \in (G_1 \times G_2)_{0}} \operatorname{Sym}\big(\hat{f}\big)(\tilde{\eta},\tilde{\nu}) \\
\hphantom{m_{\mu,\nu}}{} = \frac{|G_1 \cap G_2|}{d^2|\Lambda_R/\Lambda'|^2} \sum_{\tilde{\nu} \in (G_1 \times G_2)_{\nu}} \operatorname{Sym}\big(\hat{f}\big)(\tilde{\mu},\tilde{\nu}) \cdot \delta_{\operatorname{Sym}(f)(\tilde{\nu},\_)|_{G_1\cap G_2}=1}=(*).
	\end{gather*}
Fix as above a representative $\tilde{\nu}$ of the f\/iber of $\nu$, i.e., $\tilde{\nu} \in (G_1 \times G_2)_{\nu}$ such that $\operatorname{Sym}(f)(\tilde{\nu},\_)|_{G_1\cap G_2}$ $=1$ holds. Two elements fulf\/illing this property dif\/fer by an element in the subgroup $\operatorname{Rad}_0 \leq G_1 \cap G_2$, thus
\begin{gather*}
(*)=\frac{|G_1 \cap G_2|}{d^2|\Lambda_R/\Lambda'|^2} \operatorname{Sym}\big(\hat{f}\big)(\tilde{\mu},\tilde{\nu}) \sum_{\tilde{\xi} \in \operatorname{Rad}_0} \operatorname{Sym}\big(\hat{f}\big)(\tilde{\xi},\tilde{\nu}) \cdot \delta_{\operatorname{Sym}(f)(\tilde{\nu},\_)|_{G_1\cap G_2}=1} \\
\hphantom{(*)}{} = \frac{|G_1 \cap G_2||\operatorname{Rad}_0|}{d^2|\Lambda_R/\Lambda'|^2} \operatorname{Sym}\big(\hat{f}\big)(\tilde{\mu},\tilde{\nu}) \cdot \delta_{\operatorname{Sym}(\hat{f})(\tilde{\nu},\_)|_{G_1\cap G_2}=1}\, \delta_{\operatorname{Sym}(\hat{f})(\tilde{\mu},\_)|_{\operatorname{Rad}_0}=1}.
\end{gather*}
Since $m$ is symmetric, we have
\begin{gather*}
m_{\mu,\nu}=\frac{|G_1 \cap G_2||\operatorname{Rad}_0|}{d^2|\Lambda_R/\Lambda'|^2} \operatorname{Sym}\big(\hat{f}\big)(\tilde{\mu},\tilde{\nu}) \cdot \delta_{\operatorname{Sym}(\hat{f})(\tilde{\nu},\_)|_{G_1\cap G_2}=1} \delta_{\operatorname{Sym}(\hat{f})(\tilde{\mu},\_)|_{G_1\cap G_2}=1} \\
\hphantom{m_{\mu,\nu}}{} =\frac{|G_1 \cap G_2||\operatorname{Rad}_0|}{d^2|\Lambda_R/\Lambda'|^2} \operatorname{Sym}_G\big(\hat{f}\big)(\mu,\nu) \delta_{\operatorname{Rad}_\mu \neq \varnothing}\delta_{\operatorname{Rad}_\nu \neq \varnothing}
\end{gather*}
and this is invertible if an only if $\operatorname{Rad}_0 \cong \operatorname{Rad}\big(\operatorname{Sym}_G\big(\hat{f}\big)\big)=0$.
\end{proof}

\subsection[Factorizability for symmetric $R_0(f)$]{Factorizability for symmetric $\boldsymbol{R_0(f)}$}\label{section5.2}

For $R_0 = \sum_{\mu,\nu} f(\mu,\nu) K_\mu \otimes K_\nu$ being the Cartan part of an $R$-matrix, assume that $\hat{f}=|G|f$ on~$G$ is symmetric. We have shown in Example~\ref{ex:RadSymf} that factorizability is equivalent to~$|G|$ being odd.

We now want to give a necessary and suf\/f\/icient condition for this:

\begin{Lemma}\label{lm:conditionFor|G|odd}
	Let $\Lambda_R \subseteq \Lambda \subseteq \Lambda_W$ be an arbitrary intermediate lattice for a certain irreducible root system. Then the order of the group $G=\Lambda/\operatorname{Cent}_{\Lambda_R}(\Lambda)$ is odd if and only if both of the following conditions are satisfied:
	\begin{enumerate}\itemsep=0pt
		\item[$1)$] $|\Lambda/\Lambda_R|$ is odd,
		\item[$2)$] $\ell$ is either odd or $(\ell \equiv 2$ {\rm mod} $4$, $\g=B_n$, $\Lambda=\Lambda_R)$ including $A_1$.
	\end{enumerate}
\end{Lemma}

\begin{proof}We saw that in all our cases, there exists an isomorphism
\begin{gather*}
\Lambda/\Lambda_R \cong \operatorname{Cent}_{\Lambda}(\Lambda_R)/\operatorname{Cent}_{\Lambda_R}(\Lambda).
\end{gather*}
Moreover, from Lemma~\ref{lm:ExplicitFormOfCentralizers} we know that $|\Lambda/\operatorname{Cent}_{\Lambda}(\Lambda_R)|=\det (D_\ell)$, where $D_\ell$ was the diagonal matrix $\operatorname{Diag}\big( \frac{\ell}{\operatorname{gcd}(\ell,d_i^\Lambda)}\big) )$ with $d_i^\Lambda$ being the invariant factors of the lattice $\Lambda$ (i.e., the diagonal entries of the Smith normal form of a basis matrix of $\Lambda$). Thus,
\begin{gather*}|G|= |\Lambda/\operatorname{Cent}_{\Lambda_R}(\Lambda)| = |\Lambda/\operatorname{Cent}_{\Lambda}(\Lambda_R)||\operatorname{Cent}_{\Lambda}(\Lambda_R)/\operatorname{Cent}_{\Lambda_R}(\Lambda)| \\
\hphantom{|G|}{} =	|\Lambda/\operatorname{Cent}_{\Lambda}(\Lambda_R)||\Lambda/\Lambda_R| = \det (D_\ell) |\Lambda/\Lambda_R|
= \prod_{i=1}^n \frac{\ell}{\operatorname{gcd}(\ell,d_i^\Lambda)} |\Lambda/\Lambda_R|.
\end{gather*}
Clearly, this term is odd if $\ell$ and $|\Lambda/\Lambda_R|$ are odd. In the case ($\ell \equiv 2$ mod $4$, $\g=B_n$, $\Lambda=\Lambda_R$), the Smith normal form $S_R$ of the basis matrix $A_R$ is given by $2\cdot \operatorname{id}$. Thus, $|G|$ is odd in this case. On the other hand, let $|G|$ be odd:

We f\/irst consider the case $\ell$ \textit{even}. A necessary condition for $|\Lambda/\Lambda'|$ odd is that the multipli\-ci\-ty~$m_\ell$ of the prime~$2$ in $\prod\limits_{i=1}^n \frac{\ell}{\operatorname{gcd}(\ell,d_i^\Lambda)}$ is at most the multiplicity $m_{\pi_1}$ of the prime~$2$ in~$|\pi_1|$. We check this condition for rank $n>1$:
\begin{itemize}\itemsep=0pt
\item For $\g$ simply-laced (or triply-laced $\g=G_2$) we have all $d_i=1$, hence $n\,|\, m_\ell$ (equality for $\ell=2$ ${\rm mod}~4$). The cases $D_n$ with $m_{\pi_1}=2$ have rank $n\geq 4$, all others except $A_n$ have $m_{\pi_1}=0,1$, so the necessary condition $m_\ell\leq m_{\pi_1}$ is never fulf\/illed. The cases $A_n$ have $2^{m_{\pi_1}}|(n+1)\leq (m_\ell+1)\stackrel{!}{\leq} (m_{\pi_1}+1)$ which can only be true in rank $n=1$ treated above.
\item For $\g$ doubly-laced of rank $n>1$, we always have always $m_{\pi_1}=0,1$ but $m_\ell$ can be considerably smaller than above, namely for $\ell=2$ ${\rm mod}~4$ equal to the number of short simple roots $d_{\alpha_i}=1$ (otherwise	$m_\ell$ again increases by $n$ for every factor $2$ in~$\ell$), hence the necessary condition $m_\ell\leq m_{\pi_1}$ can be fulf\/illed only for $B_n$ (which would also include $A_1$ above for $n=1$). More precisely, since $m_\ell=m_{\pi_1}$ and the decomposition for $\Lambda/\Lambda'$ has an additional	factor $|\Lambda/\Lambda_R|$, it can only be odd for $\Lambda=\Lambda_R$.
	\end{itemize}
On the other hand, if $\ell$ is \textit{odd}, then the whole product term is odd. But since $|G|$ was assumed to be odd, also $|\Lambda/\Lambda'|$ must be odd.
\end{proof}

\begin{Corollary}\label{cor:FactForLambda=LambdaR}
	Let $\Lambda=\Lambda_R$. In the previous section we have seen that $\hat{f}=q^{-(\cdot,\cdot)}$ gives always an $R$-matrix in this case. By the proof of the previous lemma, we have
	\begin{gather*}
		\operatorname{Rad}_0 \cong \prod_{i=1}^n \Z_{\gcd\big(2,\frac{\ell}{\gcd(\ell,d_i^R)} \big) },
	\end{gather*}
	where the $d_i^R$ denote the invariant factors of $\Lambda_W^\vee/\Lambda_R$.
\end{Corollary}

\subsection[Factorizability for $D_{2n}$, $R_0$ antisymmetric]{Factorizability for $\boldsymbol{D_{2n}}$, $\boldsymbol{R_0}$ antisymmetric}\label{section5.3}
The split case $\g=D_{2n}$, $G=G_1 \times G_2$ is clearly factorizable, so the only remaining case for which we have to check factorizabilty is $\g=D_{2n}$, $\Lambda=\Lambda_W$ for $\hat{f}$ being not symmetric. We know that in this case, the corresponding form $g$ on $\Lambda/\Lambda_R$ is uniquely def\/ined by a $2\times 2$-matrix $K \in \mathfrak{gl}(2,\F_2)$, s.t.\ $g(\lambda_{2(n-1)+i},\lambda_{2(n-1)+j})=\exp \big( \frac{2 \pi i K_{ij}}{2} \big)$ for $i,j\in \{1,2\}$. From this we see that if $g$ is not symmetric, it must be antisymmetric, i.e., $g(\mu,\nu)=g(\nu,\mu)^{-1}$. Thus, the following lemma applies in this case, and hence there are no factorizable $R$-matrices for $D_{2n}$, $\Lambda=\Lambda_W$.
\begin{Lemma}\label{lm:FactForAsymmPairing}For $\g$ simply-laced and $\Lambda=\Lambda_W$, let $\hat{f}=q^{-(\cdot,\cdot)}g\colon G \times G \to \C^\times $ be a non-degenerate form as in Theorem~{\rm \ref{thm:solutionsgrpeq}} and Lemma~{\rm \ref{lm:NondegGroupPairing}}, s.t.\ the form $g\colon\pi_1 \times \pi_1 \to \C^\times$ is asymmetric. Then,
	\begin{gather*}		
		\operatorname{Rad}_0 \cong \bigoplus_{i=1}^{n} \Z_{\gcd(2,\ell d_i^R)},
	\end{gather*}
	where the $d_i^R$ denote the invariant factors of $\pi_1$. In particular, $\operatorname{Rad}_0=0$ holds if and only if $\gcd(2, \ell|\pi_1|)=1$.
\end{Lemma}

\begin{proof}We recall the def\/inition of $\operatorname{Rad}_0\big(\operatorname{Sym}_G\big(\hat{f}\big)\big) $ in this case:
\begin{gather*}
\operatorname{Rad}_0(\operatorname{Sym}_G(\hat{f})) = \big\{ \mu \in G \,|\, f(\nu,\mu)^{-1}=f(\mu,\nu) \ \forall \, \nu \in G \big\} \\
\hphantom{\operatorname{Rad}_0(\operatorname{Sym}_G(\hat{f}))}{}
= \big\{ \mu \in G \,|\, q^{(\nu,\mu)}g(\nu,\mu)^{-1}=q^{-(\mu,\nu)}g(\mu,\nu) \ \forall \, \nu \in G \big\} \\
\hphantom{\operatorname{Rad}_0(\operatorname{Sym}_G(\hat{f}))}{}
= \big\{ \mu \in G \,|\, q^{(\nu,\mu)}=q^{-(\mu,\nu)} \ \forall \, \nu \in G \big\} \\
\hphantom{\operatorname{Rad}_0(\operatorname{Sym}_G(\hat{f}))}{}
= \big\{ \mu \in G \,|\, q^{(2\mu,\nu)}=1 \ \forall \, \nu \in G \big\} \\
\hphantom{\operatorname{Rad}_0(\operatorname{Sym}_G(\hat{f}))}{}
= \big\{ \mu \in G \,|\, 2\mu \in \operatorname{Cent}_{2\Lambda_W}(\Lambda_W)/2\operatorname{Cent}_{\Lambda_R}(\Lambda_W) \big\} = (\ast) .
	\end{gather*}
For $\g$ is simply-laced, we have $\Lambda_W = \Lambda_W^\vee$, thus
\begin{gather*}
(\ast)\cong \operatorname{Cent}_{2\Lambda_W}(\Lambda_W)/2\operatorname{Cent}_{\Lambda_R}(\Lambda_W)
= (2 \Lambda_W \cap \ell A_R \Lambda_W)/2\ell A_R \Lambda_W \\
\hphantom{(\ast)}{} = P_R \operatorname{Diag}\big(\operatorname{lcm}\big(2,\ell d_i^R\big)\big)\Lambda_W / P_R 2 \ell S_R \Lambda_W
= \Lambda_W / \operatorname{Diag}\big(\gcd\big(2,\ell d_i^R\big)\big) \Lambda_W.
	\end{gather*}
	This proves the claim.
\end{proof}
\subsection{Transparent objects in non-factorizable cases}\label{section5.4}
In this section, we determine the transparent objects in the representation category of $u_q(\g,\Lambda)$ with our $R$-matrix given by $R_0\bar{\Theta}$ and $R_0=\frac{1}{|\Lambda/\Lambda'|}\sum\limits_{\mu,\nu\in\Lambda/\Lambda'} \hat{f}$ with $\hat{f}$ a group pairing $\Lambda_1/\Lambda'\times \Lambda_2/\Lambda'\to \C^\times$.
\begin{Definition}
	Let $\mathcal{C}$ be a braided monoidal category with braiding $c$. An object $V \in \mathcal{C}$ is called \emph{transparent} if the double braiding $c_{W,V}\circ c_{V,W}$ is the identity on $V \otimes W$ for all $W \in \mathcal{C}$.
\end{Definition}

The following theorem by Shimizu gives a very important characterization of factorizable categories:

\begin{Theorem}[{\cite[Theorem 1.1]{Shi16}}]\label{thm_Shimizu}
A braided finite tensor category is factorizable if and only if the transparent objects are direct sums of finitely many copies of the unit object.
\end{Theorem}

\begin{Corollary}In particular, for a Hopf algebra $H$ the representation category $H-\mathsf{mod}_{fd}$ is factorizable if and only if the transparent objects are multiples of the trivial representation and vice versa.
\end{Corollary}

Since in our cases $\Lambda_1\neq \Lambda_2$ can only appear in $D_{2n}$, and we know those are factorizable, we shall in the following restrict ourselves to the case $\Lambda_1=\Lambda_2=\Lambda$. The proof below works also in the more general case, but requires more notation. As usual we f\/irst reduce the Hopf algebra question to the group ring and then solve the group theoretical problem.

\begin{Lemma}If a $u_q(\g)$-module $V$, with a highest-weight vector $v$ and $K_\mu v=\chi(K_\mu)v$, is a~transparent object, then necessarily the $1$-dimensional $\Lambda/\Lambda'$-module $\C_\chi$ is a transparent object over the Hopf algebra $\C[\Lambda/\Lambda']$ with $R$-matrix~$R_0$. If $V$ is $1$-dimensional, then~$V$ is transparent if and only if~$\C_\chi$ is.
\end{Lemma}
\begin{proof}	Let $V$ be transparent. For every $\psi\colon \Lambda/\Lambda'\to\C^\times$ we have another f\/inite-dimensional module $W:=u_q(\g)\otimes_{u_q(\g)^+} \C_\psi$ with highest weight vector $w=1\otimes 1_\psi$ which we can test this assumption against
	\begin{gather*}
		c^2\colon \ V\otimes W\to W\otimes V\to V\otimes W.
	\end{gather*}
	We calculate the ef\/fect of $c^2$ on the highest-weight vectors $v\otimes w$:
\begin{gather*}
	c^2(v\otimes w)=\tau_{W\otimes V} R_0\bar{\Theta} \tau_{V\otimes W} R_0\bar{\Theta} (v\otimes w).
\end{gather*}
Because $v$, $w$ were assumed highest-weight vectors, the $\bar{\Theta}$ act trivially. Hence follows that~$\C_\chi$,~$\C_{\psi}$ have a trivial double braiding over the Hopf algebra $\C[\Lambda/\Lambda']$ with $R$-matrix $R_0$. Because we could achieve this result for any $\psi$ this means that $\C_\chi$ is transparent as asserted.
	
Now, let $V=\C_\chi$ be $1$-dimensional over $u_q(\g)$ and transparent over $\C[\Lambda/\Lambda']$, and let $w$ be any element in any module $W$, then again the two $\Theta$ act trivially, one time because $v=1_\chi$ is a highest weight vector, and one time because it is also a lowest weight vector. But if the double-braiding of $v=1_\chi$ with any element $w$ is trivial, then $V=\C_\chi$ is already transparent over~$u_q(\g)$.
\end{proof}

\begin{Lemma}\label{lm_transparent}$\C_\chi$ is a transparent object over the Hopf algebra $\C[\Lambda/\Lambda']$ with $R$-matrix $R_0$ iff it is an $f$-transformed of the radical of~$\operatorname{Sym}_G\big(\hat{f}\big)$, i.e.,
\begin{gather*}
\chi(\mu)=f(\mu,\xi),\qquad \xi\in \operatorname{Rad}_0.
\end{gather*}
\end{Lemma}
\begin{proof}Since $f$ is nondegenerate, we can assume $\chi(\mu)=f(\mu,\xi)$ and wish to prove $\C_\chi$ is transparent if\/f $\xi\in \operatorname{Rad}_0$. We test transparency against any module $\C_\psi$ and also write $\psi(\mu)=f(\lambda,\mu)$ (note the order of the argument). We evaluate the double-braiding on $1_\chi\otimes 1_\psi$ and get the following scalar factor, which needs to be $=1$ for all $\psi$ in order to make~$\C_\chi$ transparent:
\begin{gather*}
\frac{1}{|G|^2}\sum_{\mu,\nu} \chi(\mu)\psi(\nu)\sum_{\substack{\mu_1+\mu_2=\mu \\ \nu_1+\nu_2=\mu}}\operatorname{Sym}\big(\hat{f}\big)((\mu_1,\mu_2),(\nu_1,\nu_2))\\
\qquad{} =\frac{1}{|G|^2}\sum_{\mu,\nu} f(\mu,\xi)f(\lambda,\nu)\sum_{\substack{\mu_1+\mu_2=\mu \\ \nu_1+\nu_2=\mu}}f(\mu_1,\nu_1)f(\nu_2,\mu_2)\\
\qquad{} =\frac{1}{|G|^2}\sum_{\mu,\nu} f(\mu,\xi)f(\lambda,\nu)\sum_{\nu_1,\mu_1}f(\mu_1,\nu_1) f(\nu,\mu)f^{-1}(\nu_1,\mu)f^{-1}(\nu,\mu_1)f(\nu_1,\mu_1)\\
\qquad{} =\frac{1}{|G|}\sum_{\nu}f(\lambda,\nu)	\sum_{\nu_1,\mu_1}f(\mu_1,\nu_1)\delta_{\xi=-\nu+\nu_1}f^{-1}(\nu,\mu_1)f(\nu_1,\mu_1)\\
\qquad{} =\frac{1}{|G|}\sum_{\nu}f(\lambda,\nu)	\sum_{\mu_1}f(\mu_1,\xi+\nu)f(\xi,\mu_1) =f^{-1}(\lambda,\xi)f^{-1}(\xi,\lambda)=\operatorname{Sym}_G\big(\hat{f}\big)(\lambda,\xi).		
\end{gather*}
This scalar factor of the double braiding is equal $+1$ for all $\lambda$ (and hence all $\C_\psi$) if\/f $\xi\in \operatorname{Rad}_0$ as asserted.
\end{proof}

The previous two lemmas combined imply that any irreducible transparent $u_q(\g)$-module has necessarily the characters $\chi(\mu)= f(\mu,\xi)$, $\xi\in\operatorname{Rad}_0$ as highest-weights, and conversely if such a character $\chi$ gives rise to $1$-dimensional $u_q(\g)$-modules (i.e., $\chi|_{2\Lambda_R}=1$), then these are guaranteed transparent objects. Hence the f\/inal step is to give more closed expressions for the $f$-transformed characters $\chi$ of the radical depending on the case and check the $1$-dimensionality condition.

In all cases where $f$ is symmetric we have seen in Example~\ref{ex:RadSymf} that $\operatorname{Rad}_0\big(\operatorname{Sym}_G\big(\hat{f}\big)\big)$ is the $2$-torsion subgroup of $\Lambda/\Lambda'$, so in these cases $\chi$ gives rise to a $1$-dimensional object.
\begin{Corollary}If $f$ is symmetric $($true for all cases except $D_{2n})$ then the transparent objects are all $1$-dimensional $\C_\chi$ where the characters $\chi$ are the $f$-transformed of the elements in the radical of the bimultiplicative form $\operatorname{Sym}\big(\hat{f}\big)|_{G}$ on $G=\Lambda/\Lambda'$. In particular the group of transparent objects is isomorphic to this radical as an abelian group.
\end{Corollary}	
\begin{Corollary}In the case of symmetric $f$ $($all cases except $D_{2n})$ the fact that $\operatorname{Rad}_0$ is the $2$-torsion of $\Lambda/\Lambda'$ and $f$-transformation is a group isomorphism shows:
	
The group $T$ of transparent objects consists of $\C_\chi$ where $\chi|_{2\Lambda}=1$, i.e., the two-torsion of the character group.
\end{Corollary}
 The remaining case in $D_{2n}$ with $f$ nonsymmetric and has been done by hand in Lemma~\ref{lm:FactForAsymmPairing}.

\section{Quantum groups with a ribbon structure}\label{section6}

In \cite[Theorem 8.23]{Mue98b+}, the existence of ribbon structures for $u_q(\g,\Lambda)$ is proven. In this section we construct a ribbon structure for all cases. In the proof, we use several auxiliary results from~\cite{Mue98b+}.
\begin{Theorem}	Let $u_q(\g,\Lambda)$ be quasitriangular Hopf algebra, with an $R$-matrix satisfying the conditions in Theorem~{\rm \ref{thm:R0}} and let $u:=S(R_{(2)})R_{(1)}$. Then $v:=K_{\nu_0}^{-1}u$ is a ribbon element in~$u_q(\g,\Lambda)$.
\end{Theorem}
\begin{proof}
We consider the natural $\mathbb{N}_0[\alpha_i \,| \, i \in I]$-grading on 	the Borel parts $u^\pm:=u_q(\mathfrak{g},\Lambda)^\pm$ \cite{Lus93}. Since $u^\pm$ is f\/inite-dimensional, there exists a maximal $\nu_0 \in \mathbb{N}_0[\alpha_i \,| \, i \in I]$, s.t.\ the homogeneous component $u_{\nu_0}^\pm$ is non-trivial. More explicitly $\nu_0$ is of the form
	\begin{gather*}		
		\nu_0= \sum_{\alpha \in \Phi^+} (\ell_\alpha -1) \alpha,
	\end{gather*}
where $\ell_\alpha:= \frac{\ell}{\operatorname{gcd}(\ell,2d_\alpha)}$.

\looseness=-1 Using the formulas $u=\big(\sum f(\mu,\nu)K_{\mu+\nu}\big)^{-1} \vartheta$ and $S(u)=\big(\sum f(\mu,\nu) K_{\mu+\nu}\big)^{-1}S(\vartheta)$, where $\vartheta=\sum \bar{\Theta}^{(2)}S^{-1}\big(\bar{\Theta}^{(2)}\big)$, M\"uller proves the formula $K_{-\nu_0}^2=u^{-1}S(u)$. Using the fact that $u$ commutes with all grouplike elements, this implies $v^2=uS(u)$.	In order to show that $v$ is central, we f\/irst show that $K_{\nu_0+2\rho}^{-1}$ is a central element. By the $K,E$-relations, this is equivalent to
\begin{gather*}
\nu_0 + 2\rho \in \operatorname{Cent}_{\Lambda}(\Lambda_R),
\end{gather*}
where $\rho=\frac{1}{2} \sum\limits_{\alpha \in \Phi^+}\,\alpha$ is the Weyl vector.
		
	We calculate directly that this is always the case:
	\begin{gather*}
	(\nu_0 + 2\rho,\beta)=q^{\sum\limits_{\alpha\in\Phi^+} (\ell_\alpha-1+1) (\alpha,\beta)}
	=q^{\ell\sum\limits_{\alpha\in\Phi^+}\frac{1}{\operatorname{gcd}(\ell,2d_\alpha)}\cdot 2d_\alpha(\alpha^\vee,\beta)}=1.
	\end{gather*}
	Since $K_{2\rho}ux=xK_{2\rho}u$ holds for all $x \in u_q(\g,\Lambda)$ (see \cite[Lemmas 8.22 and 8.19]{Mue98b+}), we have
	\begin{gather*}
		vx=K_{\nu_0}^{-1}ux=K^{-1}_{\nu_0+2\rho}K_{2\rho}ux
		=K^{-1}_{\nu_0+2\rho}xK_{2\rho}u=xK^{-1}_{\nu_0+2\rho}K_{2\rho}u=xv,
	\end{gather*}
	hence $v$ is central.
\end{proof}

\section{Open questions}\label{section7}

\begin{Question}It was surprising to us that the case $D_{2n}=\mathfrak{so}_{4n}(\C)$ has so many more solutions that the other cases, in particular with non-symmetric~$R_0$, due to the non-cyclic fundamental group. Do these additional modular tensor categories appear elsewhere? Does the non-symmetry have interesting implications on the category? 	
\end{Question}

\begin{Question}Our procedure would be similarly possible for any diagonal Nichols algebra. The Lusztig ansatz can in these cases be found in~{\rm \cite{AY13}}.
\end{Question}
	
\begin{Question}\label{q:modularize} In each case where $u_q(\g,\Lambda),R$ is not factorizable, we can modularize $($see {\rm \cite{Bru00})} the corresponding representation category and get a modular tensor category, which should be representations over some ``quasi-quantum group'' $u_q(\g,\tilde{\Lambda},\omega),R$ which is a quasi-Hopf algebra where the group ring $\C[\tilde{\Lambda}]$ is deformed by a $3$-group-cocycle~$\omega$. Can we describe this quasi-Hopf algebra in a closed form? Moreover, is every factorizable quasi-quantum group the modularization of a quasi-triangular quantum group from our list?
\end{Question}

More technically:
	
\begin{Question} The centralizer transfer map $A_\ell$ in Definition~{\rm \ref{def:matrix_A_l}} $($and correspondingly the form~$a_\ell)$ had a very general characterization, but we could only prove existence by a construction using the classification of simple Lie algebras $($and distinguishing three cases$)$. We strongly suspect that these maps exist under rather general assumptions.
	
Also the result Theorem~{\rm \ref{thm:solutionsgrpeq}} from our previous article~{\rm \cite{LN14b}} has only been proven there for cyclic groups $($and by hand for $\Z_2\times \Z_2)$ although we strongly suspect it holds for every abelian group.
\end{Question}

\subsection*{Acknowledgements}

Both authors thank Christoph Schweigert for helpful discussions and support. They also thank the referees, who gave a relevant contribution to improve the article with their comments. The f\/irst author was supported by the DAAD P.R.I.M.E program funded by the German BMBF and the EU Marie Curie Actions as well as the Graduiertenkolleg RTG 1670 at the University of Hamburg. The second author was supported by the Collaborative Research Center SFB 676 at the University of Hamburg.	


\pdfbookmark[1]{References}{ref}
\LastPageEnding

\end{document}